\documentclass[a4paper]{elsarticle}
\usepackage{amsmath,amssymb,amsthm,bm}
\usepackage{enumerate}
\usepackage{url}
\theoremstyle{plain}
   \newtheorem{theorem}{Theorem}[section]
   \newtheorem{proposition}[theorem]{Proposition}
   \newtheorem{corollary}[theorem]{Corollary}
   \newtheorem{lemma}[theorem]{Lemma}
\theoremstyle{definition}
   \newtheorem{definition}[theorem]{Definition}
   
   \newtheorem{example}[theorem]{Example}
\theoremstyle{remark}
   \newtheorem{remark}[theorem]{Remark}
\newcommand{\calA}{\mathcal{A}}

\newcommand{\calD}{\mathcal{D}}
\newcommand{\calE}{\mathcal{E}}
\newcommand{\calF}{\mathcal{F}}
\newcommand{\calL}{\mathcal{L}}
\newcommand{\calLpq}{\mathcal{L}^{p,q}}
\newcommand{\calLq}{\mathcal{L}^{q}}
\newcommand{\calLinfty}{\mathcal{L}^{\infty}}
\newcommand{\calM}{\mathcal{M}}
\newcommand{\calS}{\mathcal{S}}

\newcommand{\bbN}{\mathbb{N}}
\newcommand{\bbR}{\mathbb{R}}
\newcommand{\exR}{\overline{\bbR}}
\newcommand{\seqi}{{i\in\bbN}}
\newcommand{\seqj}{{j\in\bbN}}
\newcommand{\seql}{{l\in\bbN}}
\newcommand{\seqm}{{m\in\bbN}}
\newcommand{\seqn}{{n\in\bbN}}
\newcommand{\seqk}{{k\in\bbN}}
\newcommand{\ninfty}{{n\to\infty}}
\newcommand{\kinfty}{{k\to\infty}}

\newcommand{\ep}{\varepsilon}
\newcommand{\eset}{\emptyset}

\newcommand{\mconv}{\stackrel{\!\mu}{\longrightarrow}}
\newcommand{\us}[1]{{\upshape #1}}
\newcommand{\inftyn}[1]{{\|#1\|_\infty}}
\newcommand{\pqn}[1]{\|#1\|_{p,q}}

\newcommand{\qn}[1]{\|#1\|_{q}}
\newcommand{\pqnm}[2]{\|#1\colon\!#2\|_{p,q}}

\newcommand{\qnm}[2]{\|#1\colon\!#2\|_{q}}
\newcommand{\pinftyn}[1]{\|#1\|_{p,\infty}}
\newcommand{\pinftynm}[2]{\|#1\colon\!#2\|_{p,\infty}}
\newcommand{\dsn}[1]{\|#1\|_{0}}
\newcommand{\dsnm}[2]{\|#1\colon\!#2\|_{0}}
\newcommand{\norm}[1]{\|#1\|}
\newcommand{\mupq}{\mu^{q/p}}
\newcommand{\Ch}{\textup{Ch}}

\newcommand{\Sh}{\textup{Sh}}
\begin{document}
\allowdisplaybreaks
%\reversemarginpar
\begin{frontmatter}
%
% Title
%
\title{The completeness and separability of function spaces in nonadditive measure theory\tnoteref{t1}}
\tnotetext[t1]{This work was supported by JSPS KAKENHI Grant Number 20K03695.}
%
% Name
%
\author[1]{Jun Kawabe\corref{cor1}}
\ead{jkawabe@shinshu-u.ac.jp}
\author[2]{Naoki Yamada}
\ead{naokiy900@gmail.com}
\address[1]{Faculty of Engineering, Shinshu University, 4-17-1 Wakasato, Nagano 380-8553, Japan}
\address[2]{Graduate School of Science and Technology, Shinshu University, 4-17-1 Wakasato,
Nagano 380-8553, Japan}
\cortext[cor1]{Corresponding author}
%
% Abstract
%
\begin{abstract}
For a nonadditive measure $\mu$, the space  $\calL^0(\mu)$ of all measurable functions,
the Choquet-Lorentz space $\calLpq(\mu)$, the Lorentz space
of weak type $\calL^{p,\infty}(\mu)$, the space $\calLinfty(\mu)$ of
all $\mu$-essentially bounded measurable functions,
and their quotient spaces are defined together with suitable prenorms on them.
Among those function spaces,  the Choquet-Lorentz space is defined by the Choquet integral,
while the Lorentz space of weak type is defined by the Shilkret integral.
Then the completeness and separability of those spaces are investigated.
A new characteristic of nonadditive measures, called property~(C), is introduced
to establish the Cauchy criterion for convergence in $\mu$-measure of measurable functions.
This criterion and suitable convergence theorems of the Choquet and Shilkret integrals
provide instruments for carrying out our investigation.
\end{abstract}
%
% Keyword and MSC
%
\end{frontmatter}

%\linenumbers

%
% Introduction (Intro)
%
\section{Introduction}\label{intro}
The completeness and separability of function spaces is useful and important in functional analysis.
The completeness is especially effective in guaranteeing the limit of a sequence
of functions in methods of successive approximation,
and the separability enables us to obtain constructive proofs for many theorems
that can be turned into
algorithms for use in numerical and constructive analysis.
\par
In this paper, the completeness and separability of various function spaces that appear in
nonadditive measure theory are discussed in quite generality.
The spaces we consider are the function spaces such as the space $\calL^0(\mu)$
of all measurable functions, the Choquet-Lorentz space $\calL^{p,q}(\mu)$,
the Lorentz space of weak type $\calL^{p,\infty}(\mu)$,
the space $\calL^\infty(\mu)$ of all $\mu$-essentially bounded measurable functions,
and their quotient spaces determined by a proper equivalence relation.
In those function spaces, the spaces $\calL^0(\mu)$ and $\calLinfty(\mu)$ are defined
using only a nonadditive measure $\mu$, while $\calLpq(\mu)$ and $\calL^{p,\infty}(\mu)$ are defined using so-called nonlinear integrals
such as the Choquet integral and the Shilkret integral.
The difficulty here is due to the fact that
the measures and integrals involved are nonadditive, and as a result, the natural distance
on the function spaces does not satisfy the triangle inequality in general.
\par
Our strategy to overcome this difficulty is as follows.
For instance, to show the completeness of the Lorentz space $\calL^{p,q}(\mu)$,
we first establish the Cauchy criterion for convergence in $\mu$-measure
in the space $\calF_0(X)$ of all measurable real-valued functions defined on a measurable space $(X,\calA)$
with a nonadditive measure $\mu$.
In other words, we try to find a characteristic of the measure $\mu$ such that for any $\{f_n\}_\seqn\subset\calF_0(X)$,
it follows that $\{f_n\}_\seqn$ is Cauchy for convergence in $\mu$-measure if and only if $\{f_n\}_\seqn$ converges
in $\mu$-measure.
This criterion enables us to find the limit function $f$ of a given Cauchy sequence $\{f_n\}_\seqn$
in $\calL^{p,q}(\mu)$ with respect to convergence in $\mu$-measure.
Next we show and apply suitable integral convergence theorems of the Choquet integral
to verify that the limit $f$ belongs to $\calL^{p,q}(\mu)$
and the sequence $\{f_n\}_\seqn$ converges to $f$
with respect to the distance on $\calL^{p,q}(\mu)$.
We should remark that the completeness of the space $\calL^p(\mu)=\calL^{p,p}(\mu)$
of a submodular $\mu$ and the space $\calL^\infty(\mu)$ of a subadditive $\mu$
was proved in~\cite{Denneberg} provided that $\mu$ is continuous from below.
In~\cite{P-S-R} the classical $\calL^p$ space was generalized in the framework
of the pseudo-analysis.
\par
The paper is organized as follows.
Section~\ref{pre} sets up notation and terminology. It also contains a discussion of an equivalence
relation in the space $\calF_0(X)$.
The aim of Section~\ref{Cauchy} is to find a sufficient condition
to be imposed on a nonadditive measure $\mu$
for the Cauchy criterion for convergence in $\mu$-measure to hold in the space $\calF_0(X)$.
One of such conditions can be found in~\cite{J-S-W-K-L-Y} and it is shown
that the Cauchy criterion holds in $\calF_0(X)$
if $\mu$ is continuous from below and has the pseudometric
generating property; see also~\cite{L-M-P-K}.
In this section a new characteristic of nonadditive measures,
which is weaker than the continuity from below,
is introduced to show that the Cauchy criterion for convergence
in $\mu$-measure holds in $\calF_0(X)$
if $\mu$ has this property in addition to the pseudometric generating property.
The property introduced above is called property~(C), and in Section~\ref{necessity},
this property is shown to be necessary for the Cauchy criterion to hold
in the case where $X$ is countable.
The completeness of various function spaces are discussed in Sections~\ref{MFS}--\ref{EBF}
by newly presenting some integral convergence theorems of the Choquet and Shilkret integrals.
Dense subsets and the separability of the function spaces are also discussed as related topics. 
Section~\ref{conclusion} provides a summary of our results and future tasks.
%
% Preliminaries (pre)
%
\section{Preliminaries}\label{pre}
Throughout the paper, $(X,\calA)$ is a measurable space, that is,
$X$ is a nonempty set and $\calA$ is a $\sigma$-field of subsets of $X$.
Let $\bbR$ denote the set of the real numbers and $\bbN$
the set of the natural numbers.
Let $\exR:=[-\infty,\infty]$ be the set of the extended real numbers with usual total order
and algebraic structure.
Assume that $(\pm\infty)\cdot 0=0\cdot (\pm\infty)=0$
since this proves to be convenient in measure and integration theory.
\par
For any $a,b\in\exR$, let $a\vee b:=\max\{a,b\}$ and $a\wedge b:=\min\{a,b\}$
and for any $f,g\colon X\to\exR$, let $(f\vee g)(x):=f(x)\vee g(x)$
and $(f\wedge g)(x):=f(x)\wedge g(x)$
for every $x\in X$. 
Let $\calF_0(X)$ denote the set of all $\calA$-measurable real-valued functions on $X$.
Then $\calF_0(X)$ is a real linear space with usual pointwise addition and scalar multiplication.
For any $f,g\in\calF_0(X)$, the notation $f\leq g$ means that $f(x)\leq g(x)$ for every $x\in X$.
Let $\calF_0^+(X):=\{f\in\calF_0(X)\colon f\geq 0\}$.
A function taking only a finite number of real numbers is called a \emph{simple function}.
Let $\calS(X)$ denote the set of all $\calA$-measurable simple functions on $X$.
\par
For a sequence $\{a_n\}_\seqn\subset\exR$ and $a\in\exR$, the notation
$a_n\uparrow a$ means
that $\{a_n\}_\seqn$ is nondecreasing and $a_n\to a$, and $a_n\downarrow a$ means
that $\{a_n\}_\seqn$ is nonincreasing and $a_n\to a$. 
For a sequence $\{A_n\}_\seqn\subset\calA$ and $A\in\calA$,
the notation $A_n\uparrow A$ means
that $\{A_n\}_\seqn$ is nondecreasing and $A=\bigcup_{n=1}^\infty A_n$,
and $A_n\downarrow A$
means that $\{A_n\}_\seqn$ is nonincreasing and $A=\bigcap_{n=1}^\infty A_n$.
The characteristic function of a set $A$, denoted by $\chi_A$, is the function on $X$
such that $\chi_A(x)=1$ if $x\in A$ and $\chi_A(x)=0$ otherwise.
Given two sets $A$ and $B$, let $A\triangle B:=(A\setminus B)\cup (B\setminus A)$
and $A^c:=X\setminus A$.
Let $2^X$ denote the collection of all subsets of $X$.
%
% 2.1 Nonadditive measures (measure)
%
\subsection{Nonadditive measures}\label{measure}
A \emph{nonadditive measure}\/ is a set function $\mu\colon\calA\to [0,\infty]$ such that
$\mu(\emptyset)=0$ and $\mu(A)\leq\mu(B)$ whenever $A,B\in\calA$ and $A\subset B$.
This type of set function is also called a monotone measure~\cite{W-K},
a capacity~\cite{Choquet},
or a fuzzy measure~\cite{R-A,Sugeno} in the literature.
\par
Let $\calM(X)$ denote the set of all nonadditive measures $\mu\colon\calA\to [0,\infty]$.
We say that $\mu$ is \emph{order continuous}~\cite{Drewnowski}
if $\mu(A_n)\to 0$ whenever $A_n\downarrow\eset$,
\emph{conditionally order continuous}\/ if $\mu(A_n)\to 0$ whenever $A_n\downarrow\eset$
and $\mu(A_1)<\infty$,
\emph{strongly order continuous}~\cite{J-K-W} if $\mu(A_n)\to 0$
whenever $A_n\downarrow A$ and $\mu(A)=0$,
\emph{continuous from above}\/ if $\mu(A_n)\to\mu(A)$ whenever $A_n\downarrow A$,
\emph{continuous from below}\/ if $\mu(A_n)\to\mu(A)$ whenever $A_n\uparrow A$,
and \emph{continuous}\/ if it is continuous from above and from below.
If $\mu$ is continuous from above, then it is strongly order continuous, hence order continuous.
If $\mu$ is order continuous, then it is conditionally order continuous, but the converse
does not hold even for the Lebesgue measure on the real line.
\begin{remark}
In measure theory the notion of continuity from above is usually defined by adding the
condition $\mu(A_1)<\infty$ in such a way that the Lebesgue measure
satisfies this continuity.
A nonadditive measure $\mu$ with the property that
$\mu(A_n)\to\mu(A)$ whenever $A_n\downarrow A$ and $\mu(A_1)<\infty$
is called \emph{conditionally continuous from above}.
The continuity from above obviously implies the conditional continuity from above. 
Thus, in nonadditive measure theory, there are two types of continuity from above,
which are clearly distinguished and both are indispensable.
This is because we often deal with nonadditive measures that are deformations
of the Lebesgue measure, so called distorted measures~\cite{Denneberg,Pap}
or quasi-measures~\cite{W-K} in nonadditive measure theory. 
In fact, the distorted nonadditive measure $\mu$, which is
defined by $\mu(A):=\tan(\pi\lambda(A)/2)$ for every Borel subset of $[0,1]$,
is continuous from above, where $\lambda$ is the Lebesgue measure
and let $\tan(\pi/2):=\infty$.
\end{remark}
\par
Following the terminology used in~\cite{W-K},
$\mu$ is called \emph{weakly null-additive}\/ if $\mu(A\cup B)=0$
whenever $A,B\in\calA$ and $\mu(A)=\mu(B)=0$,
\emph{null-additive}\/ if $\mu(A\cup B)=\mu(A)$ whenever $A,B\in\calA$ and $\mu(B)=0$,
\emph{autocontinuous from above}\/ if $\mu(A\cup B_n)\to\mu(A)$
whenever $A, B_n\in\calA$ and $\mu(B_n)\to 0$, and
\emph{autocontinuous from below}\/ if $\mu(A\setminus B_n)\to\mu(A)$
whenever $A, B_n\in\calA$ and $\mu(B_n)\to 0$.
Furthermore, we say that $\mu$ has \emph{property (S)}~\cite{Sun}
if any $\{A_n\}_\seqn\subset\calA$ with $\mu(A_n)\to 0$
has a subsequence $\{A_{n_k}\}_\seqk$
such that $\mu\left(\bigcap_{i=1}^\infty\bigcup_{k=i}^\infty A_{n_k}\right)=0$,
\emph{property ($\mbox{S}_1$)}~\cite{L-L}\/ if any $\{A_n\}_\seqn\subset\calA$
with $\mu(A_n)\to 0$ has a subsequence $\{A_{n_k}\}_\seqk$
such that $\mu\left(\bigcup_{k=i}^\infty A_{n_k}\right)\to 0$,
and the \emph{pseudometric generating property}\/ ((p.g.p.)~for short)~\cite{D-F}
if $\mu(A_n\cup B_n)\to 0$ whenever $A_n,B_n\in\calA$
and $\mu(A_n)\lor\mu(B_n)\to 0$.
It is easy to see that $\mu$ has the (p.g.p.)~if and only if for any $\ep>0$ there is a $\delta>0$
such that $\mu(A\cup B)<\ep$ whenever $A,B\in\calA$ and $\mu(A)\lor\mu(B)<\delta$.
\par
The following characteristic of nonadditive measures is also used.
We say that $\mu$ is \emph{monotone autocontinuous from above}~\cite{Rebille}\/ if
$\mu(A\cup B_n)\to\mu(A)$ whenever $A,B_n\in\calA$, $\{B_n\}_\seqn$ is nonincreasing,
and $\mu(B_n)\to 0$, \emph{monotone autocontinuous from below}~\cite{Rebille}\/ if
$\mu(A\setminus B_n)\to\mu(A)$ whenever $A,B_n\in\calA$,
$\{B_n\}_\seqn$ is nonincreasing,
and $\mu(B_n)\to 0$, and \emph{null-continuous}~\cite{U-M}
if $\mu(\bigcup_{n=1}^\infty N_n)=0$ whenever $\{N_n\}_\seqn\subset\calA$ is nondecreasing
and $\mu(N_n)=0$ for every $\seqn$.
\par
The autocontinuity from above implies the monotone autocontinuity from above,
while the autocontinuity from below implies the monotone autocontinuity from below.
If $\mu$ is monotone autocontinuous from above or from below,
then it is null-additive, hence weakly null-additive.
If $\mu$ has the (p.g.p.), then it is weakly null-additive.
Property ($\mbox{S}_1$) always implies property (S);
they are equivalent if $\mu$ is strongly order continuous.
If $\mu$ is continuous from below and has the (p.g.p.), then it has property ($\mbox{S}_1$);
see~\cite[Corollary~1]{J-W-Z-W-K}.
\par
If $\mu$ is autocontinuous from above and continuous from below,
then it is autocontinuous from below~\cite[Theorem~6.12]{W-K} and
has property (S)~\cite[Proposition~6]{J-S-W-K}
and the (p.g.p.)~\cite[Theorem~2]{J-W-Z-W-K}.
If $\mu$ is continuous from below, then it is null-continuous.
The null-continuity also follows from property (S)~\cite[Proposition~3.1]{U-M}
and from the conjunction of the weak null-additivity
and the strong order continuity~\cite[Proposition~9]{A-U-M}.
If $X$ is countable, then the null-continuity is equivalent
to property (S)~\cite[Proposition~3.2]{U-M}.
\par
A nonadditive measure $\mu$ is called \emph{subadditive}\/ if 
$\mu(A\cup B)\leq\mu(A)+\mu(B)$ for every disjoint $A,B\in\calA$,
\emph{relaxed subadditive}\/ if there is a constant $K\geq 1$ such that
$\mu(A\cup B)\leq K\left\{\mu(A)+\mu(B)\right\}$
for every disjoint $A,B\in\calA$ (in this case
$\mu$ is called \emph{$K$-relaxed subadditive}).
Every subadditive nonadditive measure is relaxed subadditive.
If $\mu$ is relaxed subadditive, then it has the (p.g.p.).
\begin{remark}
The relaxed subadditivity is also called the quasi-subadditivity according to the terminology
used in metric space theory.
\end{remark}
It is easy to see that a nonadditive measure $\mu$ has one of the characteristics
introduced above other than the subadditivity if and only if the measure $\mu^r$,
which is defined by $\mu^r(A):=\mu(A)^r$ for every $A\in\calA$ and $r>0$, has
the same characteristic. For instance, $\mu$ has the (p.g.p.)~if and
only if $\mu^r$ has the (p.g.p.).
This observation is to be useful and important when discussing the completeness and separability
of the Choquet-Lorentz space in Section~\ref{CLspace}.
\par
See ~\cite{Denneberg,Pap,W-K} for further information on nonadditive measures.
\subsection{The Choquet and Shilkret integrals}\label{integrals}
The following nonlinear integrals are widely used in nonadditive measure theory and its applications.
Let $\mu\in\calM(X)$. The \emph{Choquet integral}~\cite{Choquet, Schmeidler} is defined by
\[
\Ch(\mu,f):=\int_0^\infty\mu(\{f>t\})dt,
\]
for every $f\in\calF_0^+(X)$, where the right hand side is the Lebesgue integral.
The Choquet integral is equal to the abstract Lebesgue integral if $\mu$ is
$\sigma$-additive~\cite[Propositions~8.1 and 8.2]{Kawabe2016}.
The \emph{Shilkret integral}~\cite{Shilkret,Zhao} is defined by
\[
\Sh(\mu,f):=\sup_{t\in [0,\infty)}t\mu(\{f>t\})
\]
for every $f\in\calF_0^+(X)$.
In the above definitions the nonincreasing distribution function $\mu(\{f>t\})$ may be replaced with
$\mu(\{f\geq t\})$ without any change. 
\par
The following elementary properties of the Choquet and Shilkret integrals are 
easy to prove; see also~\cite{Kawabe2016}.
Recall that these integrals are neither additive nor even subadditive in general.
\begin{itemize}
\item Monotonicity:\ For any $f,g\in\calF_0^+(X)$, if $f\leq g$,
then $\Ch(\mu,f)\leq\Ch(\mu,g)$ and $\Sh(\mu,f)\leq\Sh(\mu,g)$.
\item Generativity:\ For any $c\geq 0$ and $A\in\calA$,
it follows that $\Ch(\mu,c\chi_A)=\Sh(\mu,c\chi_A)=c\mu(A)$.
\item Positive homogeneousness:\ For any $c\geq 0$ and $f\in\calF_0^+(X)$,
it follows that $\Ch(\mu,cf)=c\,\Ch(\mu,f)$ and $\Sh(\mu,cf)=c\,\Sh(\mu,f)$.
\item Elementariness:\ If $h\in\calS(X)$ is represented by
\[
h=\sum_{k=1}^n(c_k-c_{k-1})\chi_{A_k}=\bigvee_{k=1}^n c_k\chi_{A_k},
\]
where $\seqn$, $c_0=0<c_1<c_2<\dots<c_n<\infty$, and
$A_1\supset A_2\supset\dots\supset A_n$, then it follows that
\[
\Ch(\mu,h)=\sum_{k=1}^n (c_k-c_{k-1})\mu(A_k)\;\mbox{ and }\;\Sh(\mu,h)=\bigvee_{k=1}^n c_k\mu(A_k).
\]
\item Relaxed subadditivity:\ If $\mu$ is $K$-relaxed subadditive for some $K\geq 1$,
then for any $f,g\in\calF_0^+(X)$ it follows that
\[
\Ch(\mu,f+g)\leq 2K\bigl\{\Ch(\mu,f)+\Ch(\mu,g)\bigr\}
\]
and
\[
\Sh(\mu,f+g)\leq 2K\bigl\{\Sh(\mu,f)+\Sh(\mu,g)\bigr\}.
\]
\item Upper marginal continuity:\ For any $f\in\calF_0^+(X)$, it follows that
\[
\Ch(\mu,f)=\sup_{r>0}\Ch(\mu,f\land r)
\;\mbox{ and }\;\Sh(\mu,f)=\sup_{r>0}\Sh(\mu,f\land r).
\]
\item Transformation formula:\ Let $0<p<\infty$. For any $f\in\calF_0^+(X)$ it follows that
\[
\Ch(\mu,f^p)=\int_0^\infty pt^{p-1}\mu(\{f>t\})dt
\]
and
\[
\Sh(\mu,f^p)=\Sh(\mu^{1/p},f)^p.
\]
\end{itemize}
\subsection{Various modes of convergence of measurable functions}\label{mode}
Let $\{f_n\}_\seqn\subset\calF_0(X)$ and $f\in\calF_0(X)$.
There are several ways to define the convergence of sequences of measurable functions.
We say that $\{f_n\}_\seqn$ converges \emph{$\mu$-almost everywhere}\/ to $f$,
denoted by $f_n\to f$ $\mu$-a.e.,
if there is an $N\in\calA$ such that $\mu(N)=0$ and $f_n(x)\to f(x)$ for every $x\not\in N$.
We also say that $\{f_n\}_\seqn$ converges \emph{$\mu$-almost uniformly}\/ to $f$, denoted by
$f_n\to f$ $\mu$-a.u., if for any $\ep>0$ there is an $E_\ep\in\calA$ such that $\mu(E_\ep)<\ep$
and $f_n$ converges to $f$ uniformly on $X\setminus E_\ep$.
Another concept of convergence is not quite intuitive, but it has
some advantages in analysis.
We say that $f_n$ converges \emph{in $\mu$-measure}\/ to $f$, denoted by $f_n\mconv f$,
if $\mu(\{|f_n-f|>\ep\})\to 0$ for every $\ep>0$.
The sequence $\{f_n\}_\seqn\subset\calF_0(X)$ is simply
said to converge $\mu$-almost everywhere,
$\mu$-almost uniformly, and in $\mu$-measure, if there is a function $f\in\calF_0(X)$ such that
$f_n\to f$ $\mu$-a.e., $f_n\to f$ $\mu$-a.u., and $f_n\mconv f$, respectively.
Every sequence of measurable functions converging $\mu$-almost uniformly converges
$\mu$-almost everywhere and in $\mu$-measure to the same limit function.
\par
The three modes of convergence introduced above
require that the differences between the elements $f_n$
of the sequence and the limit function $f$ should become small in some sense as $n$ increases.
The following definition involves only the elements of the sequence.
We say that $\{f_n\}_\seqn$
is \emph{Cauchy in $\mu$-measure}\/ if for any $\ep>0$ and $\delta>0$
there is an $n_0\in\bbN$ such that $\mu(\{|f_m-f_n|>\ep\})<\delta$
whenever $m,n\in\bbN$ and $m,n\geq n_0$.
\par
The relation between convergence in measure and almost everywhere convergence
is made precise in nonadditive versions of the Lebesgue and the Riesz theorem.
The former states that any sequence converging $\mu$-almost everywhere converges in $\mu$-measure
if and only if $\mu$ is strongly order continuous~\cite[Theorem~5.2]{L-M-P-K},
and the latter states that any sequence converging in $\mu$-measure
has a subsequence converging $\mu$-almost everywhere
if and only if $\mu$ has property (S)~\cite[Theorem~5.17]{L-M-P-K}.
Furthermore, any sequence converging in $\mu$-measure has a subsequence converging $\mu$-almost
uniformly if and only if $\mu$ has property ($\mbox{S}_1$)~\cite[Theorem~4]{L-L}.
For Cauchy in $\mu$-measure sequences, it can be found in~\cite[Theorem~7.2]{L-M-P-K} that
any sequence converging in $\mu$-measure is Cauchy in $\mu$-measure if and only if
$\mu$ has the (p.g.p.).
By~\cite[Theorem~7.3]{L-M-P-K}
any Cauchy in $\mu$-measure sequence always converges in $\mu$-measure
if $\mu$ is continuous from below and has the (p.g.p.).
Thus the conjunction of the continuity of $\mu$ from below and the (p.g.p.)~is a sufficient condition
for the  Cauchy criterion to hold for convergence in $\mu$-measure of measurable functions.
\par
See a survey paper~\cite{L-M-P-K} for further information on various modes
of convergence of measurable functions in nonadditive measure theory.
%
% Equivalence relation
%
\subsection{Equivalence relation and quotient space}\label{equiv}
The quotient space of $\calF_0(X)$ is constructed by an equivalence relation
determined by a nonadditive measure $\mu$.
The proof of the following statements is routine and left it to the reader.
\begin{itemize}
\item Assume that $\mu$ is weakly null-additive. Given $f,g\in\calF_0(X)$,
define the binary relation $f\sim g$ on $\calF_0(X)$ by $\mu(\{|f-g|>c\})=0$ for every $c>0$
so as to become an equivalence relation on $\calF_0(X)$.
For every $f\in\calF_0(X)$ the equivalence class of $f$
is the set of the form $\{g\in\calF_0(X)\colon f\sim g\}$ and denoted by $[f]$.
Then the quotient space of $\calF_0(X)$ is defined by $F_0(X):=\{[f]\colon f\in\calF_0(X)\}$.
\item Assume that $\mu$ is weakly null-additive.
Given equivalence classes $[f],[g]\in F_0(X)$ and $c\in\bbR$, define addition and scalar multiplication
on $F_0(X)$ by $[f]+[g]:=[f+g]$ and $c[f]:=[cf]$.
They are well-defined, that is, they are independent of which member of an equivalence class we choose
to define them.
Then $F_0(X)$ is a real linear space.
\end{itemize}
\par
The binary relation on $\calF_0(X)$ defined above may not
be transitive unless $\mu$ is weakly null-additive; see~\cite[Example~5.1]{Kawabe2021}.
In what follows, let $S(X):=\{[h]\colon h\in\calS(X)\}$.
%
% Prenorms
%
\subsection{Prenorms}\label{prenorm}
Let $V$ be a real linear space.
A \emph{prenorm} on $V$ is a nonnegative real-valued function $\norm{\cdot}$
defined on $V$ such that $\norm{0}=0$ and $\norm{-x}=\norm{x}$ for every $x\in V$.
Then the pair $(V,\norm{\cdot})$ is called a \emph{prenormed space}.
A prenorm $\norm{\cdot}$ is called \emph{homogeneous}\/ if it follows that
$\norm{cx}=|c|\norm{x}$ for every $x\in V$ and $c\in\bbR$ and
\emph{truncated subhomogeneous}\/ if it follows that
$\norm{cx}\leq\max(1,|c|)\norm{x}$ for every $x\in V$ and $c\in\bbR$.
A \emph{seminorm}\/ is a prenorm that is homogeneous and satisfies the triangle inequality, that is,
$\norm{x+y}\leq\norm{x}+\norm{y}$ for every $x,y\in V$.
Then a \emph{norm} is a seminorm that separates points of $V$, that is,
for any $x\in V$, if $\norm{x}=0$ then $x=0$.
Following~\cite{D-D}, a prenorm $\norm{\cdot}$ is called relaxed if it
satisfies a \emph{relaxed triangle inequality}, that is, there is a constant $K\geq 1$
such that $\norm{x+y}\leq K\left\{\norm{x}+\norm{y}\right\}$ for every $x,y\in V$
(in this case, $\norm{\cdot}$ is called the \emph{$K$-relaxed prenorm}).
A \emph{quasi-seminorm}\/ on $V$ is a prenorm that is homogeneous
and satisfies a relaxed triangle inequality. Then a \emph{quasi-norm}\/ is a quasi-seminorm
that separates points of $V$.
\par
To associate with similar characteristics of nonadditive measures,
a prenorm $\norm{\cdot}$ is called \emph{weakly null-additive}\/ if $\norm{x+y}=0$
whenever $x,y\in V$ and $\norm{x}=\norm{y}=0$ and
\emph{null-additive}\/ if $\norm{x+y}=\norm{x}$
whenever $x,y\in V$ and $\norm{y}=0$.
\par
Let $(V,\norm{\cdot})$ be a prenormed space.
Let $\{x_n\}_\seqn\subset V$ and $x\in V$.
We say that $\{x_n\}_\seqn$ \emph{converges}\/ to $x$, denoted by $x_n\to x$,
if $\norm{x_n-x}\to 0$.
We may simply say that $\{x_n\}_\seqn$ converges
if the limit $x$ is not needed to specify.
The notion of a Cauchy sequence involves only the elements of the sequence and
we say that $\{x_n\}_\seqn$ is \emph{Cauchy}\/ if for
any $\varepsilon>0$ there is an $n_0\in\bbN$
such that $\norm{x_m-x_n}<\varepsilon$ whenever $m,n\in\bbN$ and $m,n\geq n_0$.
Not every converging sequence is Cauchy since prenorms satisfy neither
the triangle inequality nor its relaxed ones in general.
A subset $B$ of $V$ is called \emph{bounded}\/ if $\sup_{x\in B}\norm{x}<\infty$.
\par
A prenormed space $(V,\norm{\cdot})$ is called \emph{complete}\/
if every Cauchy sequence in $V$ converges to an element in $V$.
It is called \emph{quasi-complete}\/ if every bounded
Cauchy sequence in $V$ converges to an element in $V$.
The denseness and the separability can be defined in the same way as in ordinary normed spaces.
We say that $V$ is \emph{separable}\/ if there is a countable subset $D$ of $V$ such that
$D$ is \emph{dense} in $V$, that is,
for any $x\in V$ and $\ep>0$ there is a $y\in D$ such that $\norm{x-y}<\ep$.
\par
If the prenorm $\norm{\cdot}$ is needed to emphasize in the above terms,
then the phrase ``with respect to $\norm{\cdot}$'' is added to each term. 
%
% Section (Cauchy)
%
\section{The Cauchy criterion for convergence in measure}\label{Cauchy}
Given a $\sigma$-additive measure $\mu$, a sequence of measurable functions
converges in $\mu$-measure if and only if it is Cauchy
in $\mu$-measure~\cite[Theorems~C and~E in Chapter~IV, Section~22]{Halmos}.
This fact is referred to as the Cauchy criterion for convergence in measure
and was already extended to nonadditive measures that are continuous from below
and have the (p.g.p.).
Therefore, the discussion in this section starts with recalling some
already known results associated with the Cauchy criterion.
\begin{theorem}[{\cite[Theorems~7.2 and 7.3]{L-M-P-K} and \cite[Theorems~1 and~2]{J-S-W-K-L-Y}}]\label{base}
Let $\mu\in\calM(X)$.
\begin{enumerate}
\item[\us{(1)}] The following assertions are equivalent.
\begin{enumerate}
\item[\us{(i)}] $\mu$ has the (p.g.p.).
\item[\us{(ii)}] Any sequence $\{f_n\}_\seqn\subset\calF_0(X)$ converging in $\mu$-measure
is Cauchy in $\mu$-measure.
\end{enumerate}
\item[\us{(2)}] Assume that $\mu$ has the (p.g.p.).
If a sequence $\{f_n\}_\seqn\subset\calF_0(X)$ is Cauchy in $\mu$-measure
and has a subsequence converging in $\mu$-measure to a function $f\in\calF_0(X)$,
then $f_n\mconv f$.  
\item[\us{(3)}] Assume that $\mu$ is continuous from below and has the (p.g.p.).
Any Cauchy in $\mu$-measure sequence $\{f_n\}_\seqn\subset\calF_0(X)$ has a subsequence
converging $\mu$-almost uniformly.
\item[\us{(4)}] 
Assume that $\mu$ is continuous from below and has the (p.g.p.).
Then, for any sequence $\{f_n\}_\seqn\subset\calF_0(X)$
the following assertions are equivalent.
\begin{enumerate}
\item[\us{(i)}] $\{f_n\}_\seqn$ is Cauchy in $\mu$-measure.
\item[\us{(ii)}] $\{f_n\}_\seqn$ converges in $\mu$-measure.
\end{enumerate} 
\end{enumerate}
\end{theorem}
\par
The purpose of this section is to find a weaker characteristic of a nonadditive measure $\mu$,
under which every Cauchy in $\mu$-measure sequence of measurable functions
converges in $\mu$-measure. 
To this end we first introduce a new characteristic of nonadditive measures.
\begin{definition}\label{C-pro}
Let $\mu\in\calM(X)$. We say that $\mu$ has \emph{property~(C)}\/ if
for any sequence $\{E_n\}_\seqn\subset\calA$, it follows that
$\mu\left(\bigcup_{n=k}^\infty E_n\right)\to 0$ whenever
\[
\sup_\seql\mu\left(\bigcup_{n=k}^{k+l}E_n\right)\to 0.
\]
\end{definition}
It is easy to see that every nonadditive measure that is continuous from below has property~(C).
Other examples of nonadditive measures having property~(C) are as follows.
\begin{proposition}\label{ex-C-pro}
Let $\mu\in\calM(X)$.
\begin{enumerate}
\item[\us{(1)}] Assume that for any $\ep>0$ there is a $\delta>0$ such that condition
\[
\sup_\seqn\mu(E_n)<\delta\,\Longrightarrow\,
\mu\left(\bigcup_{n=1}^\infty E_n\right)<\ep \tag{$*$}
\]
holds for every nondecreasing sequence $\{E_n\}_\seqn\subset\calA$.
Then $\mu$ has property (C).
In particular, the above assumption holds if there is a constant $M\geq 1$ such that
\[
\mu\left(\bigcup_{n=1}^\infty E_n\right)\leq M\sup_\seqn\mu(E_n)
\]
for every nondecreasing sequence $\{E_n\}_\seqn\subset\calA$.
\item[\us{(2)}] Assume that for any $\ep>0$ there is a $\delta>0$ such that condition ($*$)
holds for every sequence $\{E_n\}_\seqn\subset\calA$.
Then $\mu$ has property~(C) and the (p.g.p.).
\item[\us{(3)}] Let $\lambda\in\calM(X)$.
Let $\theta\colon [0,\infty]\to [0,\infty]$ be a nondecreasing function with $\theta(0)=0$
that is continuous and increasing on a neighborhood of $0$.
Let $\theta\circ\lambda\colon\calA\to [0,\infty]$ is the nonadditive measure
defined by $\theta\circ\lambda(A):=\theta(\lambda(A))$ for every $A\in\calA$.
If $\lambda$ is continuous from below and has the (p.g.p.),
then $\theta\circ\lambda$ has property~(C) and the (p.g.p.).
\item[\us{(4)}] If $\mu$ has property~(C), then it is null-continuous.
\end{enumerate}
\end{proposition}
\begin{proof}
(1)\ Let $\{E_n\}_\seqn\subset\calA$ be a sequence satisfying
\begin{equation}\label{eqn:ex-C-pro1}
\sup_\seql\mu\left(\bigcup_{n=k}^{k+l}E_n\right)\to 0.
\end{equation}
Let $\ep>0$ and take $\delta>0$ such that condition ($*$) holds for
every nondecreasing sequence $\{D_l\}_\seql\subset\calA$.
For this $\delta$, by~\eqref{eqn:ex-C-pro1} there is a $k_0\in\bbN$ such that
$\sup_\seql\mu\left(\bigcup_{n=k}^{k+l}E_n\right)<\delta$ for every $k\geq k_0$.
Fix $k\geq k_0$ and let $D_l:=\bigcup_{n=k}^{k+l}E_n$ for every $\seql$.
Then $\{D_l\}_\seql$ is nondecreasing and $\sup_\seql\mu(D_l)<\delta$.
Hence it follows from condition ($*$) that
\[
\mu\left(\bigcup_{n=k}^\infty E_n\right)=\mu\left(\bigcup_{l=1}^\infty D_l\right)<\ep,
\]
which implies $\mu\left(\bigcup_{n=k}^\infty E_n\right)\to 0$.
Thus $\mu$ has property~(C). The rest is easy to prove.
\par
(2)\ It suffices to show that $\mu$ has the (p.g.p.).
Let $\ep>0$ and take $\delta>0$ such that condition ($*$) holds
for every sequence $\{E_n\}_\seqn\subset\calA$.
Let $A,B\in\calA$ and assume that $\mu(A)\lor\mu(B)<\delta$.
Let $E_1:=A$, $E_2:=B$, and $E_n:=\eset$ for every $n\geq 3$.
Then $\sup_\seqn\mu(E_n)=\mu(A)\lor\mu(B)<\delta$,
hence $\mu(A\cup B)=\mu\left(\bigcup_{n=1}^\infty E_n\right)<\ep$
by condition ($*$).
Thus $\mu$ has the (p.g.p.). 
\par
(3)\ Let $\theta$ be continuous and increasing on $[0,\delta]$ for some $\delta>0$.
By~\cite[Theorem~2.5.3]{Friedman}, the function $\theta$ has a continuous and increasing
inverse $\theta^{-1}\colon [0,\theta(\delta)]\to [0,\delta]$.
Let $\{E_n\}_\seqn\subset\calA$ be a sequence
satisfying~\eqref{eqn:ex-C-pro1} for $\mu=\theta\circ\lambda$.
Let $\ep>0$. Then there is a $k_0\in\bbN$ such that for any $k\geq k_0$ it follows that
\[
\sup_\seql\theta\left(\lambda\left(\bigcup_{n=k}^{k+l}E_n\right)\right)<\ep\land\theta(\delta).
\]
Hence the continuity of $\lambda$ from below yields
\[
\lambda\left(\bigcup_{n=k}^\infty E_n\right)
=\lambda\left(\bigcup_{l=1}^\infty\bigcup_{n=k}^{k+l}E_n\right)
=\sup_\seql\lambda\left(\bigcup_{n=k}^{k+l}E_n\right)\leq\theta^{-1}(\ep\land\theta(\delta)).
\]
Thus
\[
\theta\circ\lambda\left(\bigcup_{n=k}^\infty E_n\right)\leq\ep\land\theta(\delta)\leq\ep,
\]
which implies $\theta\circ\lambda\left(\bigcup_{n=k}^\infty E_n\right)\to 0$.
Hence $\theta\circ\lambda$ has property~(C).
\par
Let $A_n,B_n\in\calA$ for every $\seqn$ and assume that
$\theta\circ\lambda(A_n)\lor\theta\circ\lambda(B_n)\to 0$.
Then $\lambda(A_n)\lor\lambda(B_n)\to 0$, so that $\lambda(A_n\cup B_n)\to 0$
since $\lambda$ has the (p.g.p.), and finally that $\theta\circ\lambda(A_n\cup B_n)\to 0$.
Therefore, $\theta\circ\lambda$ has the (p.g.p.).
\par
(4)\ Take a nondecreasing sequence $\{N_n\}_\seqn\subset\calA$ and assume that $\mu(N_n)=0$ for
every $\seqn$. Then
\[
\sup_\seql\mu\left(\bigcup_{n=k}^{k+l}N_n\right)=\sup_\seql\mu(N_{k+l})=0
\]
for every $\seql$, hence $\mu\left(\bigcup_{n=k}^\infty N_n\right)\to 0$ by property (C).
Thus, for any $\ep>0$, there is a $k_0\in\bbN$
such that $\mu\left(\bigcup_{n=k_0}^\infty N_n\right)<\ep$,
hence $\mu\left(\bigcup_{n=1}^\infty N_n\right)<\ep$.
Letting $\ep\downarrow 0$ yields $\mu\left(\bigcup_{n=1}^\infty N_n\right)=0$.
Therefore, $\mu$ is null-continuous.
\end{proof}
Our first issue is to extend assertion (3) of Theorem~\ref{base} to nonadditive measures
having property~(C) and the (p.g.p.).
Examples~\ref{counter2} and~\ref{counter3} below
give nonadditive measures that are not continuous from below,
but have property~(C) and the (p.g.p.), so that the following theorem
is a sharpened version of (3) of Theorem~\ref{base}.
Since the idea of the proof is the same as Theorem~1 of~\cite{J-S-W-K-L-Y},
only a sketch of the proof is given here for the convenience of the reader.
\begin{theorem}\label{comp}
Let $\mu\in\calM(X)$.
If $\mu$ has property~(C) and the (p.g.p.),
then any Cauchy in $\mu$-measure sequence $\{f_n\}_\seqn\subset\calF_0(X)$
has a subsequence converging $\mu$-almost uniformly.
\end{theorem}
\begin{proof}
Since $\mu$ has the (p.g.p.),
there is a decreasing sequence $\{\delta_i\}_\seqi$ such that
\begin{itemize}
\item $\delta_0=1/2$ and $0<\delta_i<\delta_{i-1}\land\frac{1}{2^i}$ for every $\seqi$,
\item $\mu(A\cup B)<\delta_{i-1}$ whenever $\seqi$, $A,B\in\calA$,
and $\mu(A)\lor\mu(B)<\delta_i$.
\end{itemize}
\par
Take a Cauchy in $\mu$-measure sequence $\{f_n\}_\seqn\subset\calF_0(X)$.
Then there is an increasing sequence $\{n_i\}_\seqi\subset\bbN$ such that
for any $\seqi$ and $\seqn$, if $n\geq n_i$ then
\begin{equation}\label{eqn:comp1}
\mu(\{|f_n-f_{n_i}|\geq 1/2^i\})<\delta_i.
\end{equation}
For each $i\in\bbN$, let $E_i:=\{|f_{n_{i+1}}-f_{n_i}|\geq 1/2^i\}$. Then
\begin{equation}\label{eqn:comp2}
\mu\left(\bigcup_{i=k}^{r+1}E_i\right)<\delta_{k-1}
\end{equation}
for every $r\in\bbN$ and $k\in\{1,2,\dots,r+1\}$.
For each $k\in\bbN$, let $A_k:=\bigcup_{i=k}^\infty E_i$
and $E:=\bigcap_{k=1}^\infty A_k$.
Then $\{f_{n_k}(x)\}_\seqk$ is a Cauchy sequence in $\bbR$ for every $x\not\in E$.
Therefore, the $\calA$-measurable function $f\colon X\to\bbR$ can be defined by
\[
f(x):=\begin{cases}
\lim\limits_{k\to\infty}f_{n_k}(x) & \mbox{if }x\not\in E,\\
\;0 & \mbox{otherwise}.
\end{cases}
\]
\par
Fix $\seqk$. For any $\seql$, let $r:=k+l-1$. Then \eqref{eqn:comp2} yields
\[
\mu\left(\bigcup_{i=k}^{k+l}E_i\right)=\mu\left(\bigcup_{i=k}^{r+1}E_i\right)<\delta_{k-1},
\]
so that
\[
\sup_\seql\mu\left(\bigcup_{i=k}^{k+l}E_i\right)\leq\delta_{k-1}.
\]
Letting $k\to\infty$ gives $\delta_{k-1}\to 0$, hence
\[
\sup_\seql\mu\left(\bigcup_{i=k}^{k+l}E_i\right)\to 0.
\]
Since $\mu$ has property~(C), if $k\to\infty$ then
\begin{equation}\label{eqn:comp3}
\mu(A_k)=\mu\left(\bigcup_{i=k}^\infty E_i\right)\to 0.
\end{equation}
\par
Let $\ep>0$. By~\eqref{eqn:comp3} there is a $k_0\in\bbN$ such that $\mu(A_{k_0})<\ep$.
Now it is routine work to show that $f_{n_k}$ converges to $f$ uniformly on $X\setminus A_{k_0}$.
\end{proof}
\begin{definition}\label{C0}
Let $\mu\in\calM(X)$. We say that $\mu$ has \emph{property~($\mbox{C}_0$)} if
for any sequence $\{E_n\}_\seqn\subset\calA$, 
it follows that $\mu\left(\bigcap_{k=1}^\infty\bigcup_{n=k}^\infty E_n\right)=0$
whenever $\sup_\seql\mu\left(\bigcup_{n=k}^{k+l}E_n\right)\to 0$.
\end{definition}
\par
By definition, property~(C) always implies property~($\mbox{C}_0$).
It is easy to see that both properties are equivalent if $\mu$ is strongly order continuous.
Furthermore, for every $r>0$, it follows that $\mu$ has property~(C) if and only if $\mu^r$
has property~(C). The same holds for property~($\mbox{C}_0$). 
The following corollary can be proved in the same reasoning as Theorem~\ref{comp}.
\begin{corollary}\label{AE}
Let $\mu\in\calM(X)$.
If $\mu$ has property~($\mbox{C}_0$) and the (p.g.p.),
then any Cauchy in $\mu$-measure sequence $\{f_n\}_\seqn\subset\calF_0(X)$
has a subsequence converging $\mu$-almost everywhere.
\end{corollary}
Now, let us proceed to the Cauchy criterion for convergence in measure.
\begin{theorem}\label{fund}
Let $\mu\in\calM(X)$.
Assume that $\mu$ has property~(C) and the (p.g.p.). 
Then, for any sequence $\{f_n\}_\seqn\subset\calF_0(X)$
the following assertions are equivalent.
\begin{enumerate}
\item[\us{(i)}] $\{f_n\}_\seqn$ is Cauchy in $\mu$-measure.
\item[\us{(ii)}] $\{f_n\}_\seqn$ converges in $\mu$-measure.
\end{enumerate}
\end{theorem}
\begin{proof}
(i)$\Rightarrow$(ii) By Theorem~\ref{comp},
$\{f_n\}_\seqn$ has a subsequence $\{f_{n_k}\}_\seqk$ and $f\in\calF_0(X)$
such that $f_{n_k}\mconv f$. Let $\ep>0$ and $\sigma>0$.
Since $\mu$ has the (p.g.p.), there is a $\delta>0$ such that
\begin{equation}\label{eqn:fund1}
\mu(A\cup B)<\ep
\end{equation}
whenever $A,B\in\calA$ and $\mu(A)\lor\mu(B)<\delta$.
Furthermore, $\{f_n\}_\seqn$ being Cauchy in $\mu$-measure, there is an $n_0\in\bbN$ such that
for any $m,n\in\bbN$, if $m,n\geq n_0$ then
\begin{equation}\label{eqn:fund2}
\mu(\{|f_m-f_n|>\sigma/2\})<\delta.
\end{equation}
Since $f_{n_k}\mconv f$, there is a $k_0\in\bbN$ such that $n_{k_0}\geq n_0$ and
\begin{equation}\label{eqn:fund3}
\mu(\{|f_{n_{k_0}}-f|>\sigma/2\})<\delta.
\end{equation}
Therefore, it follows from \eqref{eqn:fund1}--\eqref{eqn:fund3} that if $n\geq n_0$ then
\[
\mu(\{|f_n-f|>\sigma\})\leq\mu(\{|f_n-f_{n_{k_0}}|>\sigma/2\}\cup
\{|f_{n_{k_0}}-f|>\sigma/2\})<\ep,
\]
which implies $f_n\mconv f$.
\par
(ii)$\Rightarrow$(i) It follows from (1) of Theorem~\ref{base}; see also~\cite[Theorem~7.2]{L-M-P-K}.
\end{proof}
Next we give a structural characteristic of the conjunction of property~(C) and the (p.g.p.).
The following technical lemma is prepared for this purpose and for later use.
\begin{lemma}\label{au}
Let $\mu\in\calM(X)$. Assume that $\mu$ is null-continuous and has the (p.g.p.).
Let $\{f_n\}_\seqn\subset\calF_0(X)$ and $f,g\in\calF_0(X)$.
If $f_n\mconv f$ and $f_n\to g$ $\mu$-a.u., then $f=g$ $\mu$-a.e.~and $f_n\to f$ $\mu$-a.u.
\end{lemma}
\begin{proof}
Let $N:=\{|f-g|>0\}$. Since $f_n\mconv f$ and $f_n\mconv g$, the (p.g.p.)~of $\mu$
implies that $\mu(\{|f-g|>c\})=0$ for every $c>0$.
Therefore, $\mu(N)=0$ by the null-continuity of $\mu$, hence $f=g$ $\mu$-a.e.
\par
Next we prove that $f_n\to f$ $\mu$-a.u.
Since $f_n\to g$ $\mu$-a.u., there exists $\{E_k\}_\seqk\subset\calA$ such that
$\mu(E_k)\to 0$ and $\sup_{x\not\in E_k}|f_n(x)-g(x)|\to 0$ for every $\seqk$.
For each $\seqk$, let $D_k:=E_k\cup N$. Then $\mu(D_k)\to 0$
since $\mu$ has the (p.g.p.). Furthermore, for any $\seqk$, letting $\ninfty$ gives
\[
\sup_{x\not\in D_k}|f_n(x)-f(x)|\leq\sup_{x\not\in E_k}|f_n(x)-g(x)|\to 0,
\]
which implies that $f_n\to f$ $\mu$-a.u. 
\end{proof}
\begin{proposition}\label{structure}
Let $\mu\in\calM(X)$. Assume that $\mu$ has property~(C) and the (p.g.p.).
Then $\mu$ has property~($\mbox{S}_1$), hence
it is null-continuous and has property~(S). 
\end{proposition}
\begin{proof}
Suppose, contrary to our claim, that $\mu$ does not have property ($\mbox{S}_1$).
Then, by~\cite[Theorem~4]{L-L} there exist $\{f_n\}_\seqn\subset\calF_0(X)$
and $f\in\calF_0(X)$ such that $f_n\mconv f$ and that
$\{f_n\}_\seqn$ does not have any subsequence converging $\mu$-almost uniformly to $f$.
Nevertheless,
Theorems~\ref{comp} and~\ref{fund} imply that $\{f_n\}_\seqn$ has a subsequence
$\{f_{n_k}\}_\seqk$ and $g\in\calF_0(X)$ such that $f_{n_k}\to g$ $\mu$-a.u.
Since $\mu$ is null-continuous by (4) of Proposition~\ref{ex-C-pro},
it follows from Lemma~\ref{au} that $f_{n_k}\to f$ $\mu$-a.u., which is impossible.
Hence $\mu$ has property~($\mbox{S}_1$).
See Subsection~\ref{measure} for the fact that $\mu$ has property~(S).
\end{proof}
As a supplementary result of Theorem~\ref{fund},
the notion of Cauchy in $\mu$-measure can be characterized by the convergence of subsequences.
\begin{corollary}\label{subseq}
Let $\mu\in\calM(X)$. Let $\{f_n\}_\seqn\subset\calF_0(X)$.
Assume that $\mu$ has property~(C) and the (p.g.p.).
Then the following assertions are equivalent.
\begin{enumerate}
\item[\us{(i)}] $\{f_n\}_\seqn$ is Cauchy in $\mu$-measure.
\item[\us{(ii)}] There is a function $f\in\calF_0(X)$ such that
any subsequence of $\{f_n\}_\seqn$ has a
further subsequence converging $\mu$-almost uniformly to $f$.
\item[\us{(iii)}] There is a function $f\in\calF_0(X)$ such that
any subsequence of $\{f_n\}_\seqn$ has a
further subsequence converging in $\mu$-measure to $f$.
\end{enumerate}
\end{corollary}
\begin{proof}
Observe that $\mu$ has property~($\mbox{S}_1$) by Proposition~\ref{structure}.
\par
(i)$\Rightarrow$(ii)\ Since $\{f_n\}_\seqn$ is Cauchy in $\mu$-measure,
by Theorem~\ref{fund} there is a function $f\in\calF_0(X)$ such that $f_n\mconv f$.
Assertion (ii) thus follows from~\cite[Theorem~2]{L-L}.
\par
(ii)$\Rightarrow$(iii)\ It is obvious.
\par
(iii)$\Rightarrow$(i)\ Let $\ep>0$. For each $\seqn$, let $a_n:=\mu(\{|f_n-f|>\ep\})$.
Take any subsequence $\{a_{n_k}\}_\seqk$ of $\{a_n\}_\seqn$.
Then, $\{f_{n_k}\}_\seqk$ being a subsequence of $\{f_n\}_\seqn$,
it follows from assertion (iii) that
$\{f_{n_k}\}_\seqk$ has a further subsequence $\{f_{n_{k_i}}\}_\seqi$ converging
in $\mu$-measure to $f$, hence $a_{n_{k_i}}=\mu(\{|f_{n_{k_i}}-f|>\ep\})\to 0$.
Therefore, we have $a_n\to 0$, that is, $f_n\mconv f$.
Since $\mu$ has the (p.g.p.), the sequence $\{f_n\}_\seqn$ is Cauchy in $\mu$-measure
by (1) of Theorem~\ref{base}.
\end{proof}
As the following proposition shows, property~(C) cannot be dropped in
Theorems~\ref{comp} and~\ref{fund}, Proposition~\ref{structure},
and Corollary~\ref{subseq}.
\begin{proposition}\label{counter1}
Let $X:=\bbN$ and $\calA:=2^\bbN$.
Let $\mu\colon\calA\to [0,2]$ be the nonadditive measure defined by
\[
\mu(A):=\begin{cases}
0 & \mbox{if }A=\eset,\\
\sum_{i\in A}1/2^i & \mbox{if $A$ is a nonempty finite subset of $\bbN$},\\[1mm]
1+\sum_{i\in A}1/2^i & \mbox{if $A$ is an infinite subset of $\bbN$}.
\end{cases}
\]
\begin{enumerate}
\item[\us{(1)}] $\mu$ is subadditive, null-continuous, and has the (p.g.p.)~and property~(S).
\item[\us{(2)}] $\mu$ is neither continuous from below nor order continuous.
\item[\us{(3)}] $\mu$ has neither property~(C) nor property~($S_1$). 
\item[\us{(4)}] For each $\seqn$, let $A_n:=\{1,2,\dots,n\}$ and $f_n:=\chi_{A_n}$.
Then the sequence $\{f_n\}_\seqn\subset\calF_0(X)$
is Cauchy in $\mu$-measure, but it neither converges in $\mu$-measure
nor has a subsequence converging in $\mu$-measure.
\item[\us{(5)}] For each $\seqn$, let $A_n:=\{n\}^c$ and $f_n:=\chi_{A_n}$.
Then the sequence $\{f_n\}_\seqn\subset\calF_0(X)$ converges
in $\mu$-measure to the constant function $1$
and $\{f_n\}_\seqn$ is Cauchy in $\mu$-measure,
but it does not have any subsequence converging $\mu$-almost uniformly.
\end{enumerate}
\end{proposition}
\begin{proof}
(1)\ We only show that $\mu$ has property~(S). The rest is easy to prove.
Take $\{A_n\}_\seqn\subset\calA$ and assume that $\mu(A_n)\to 0$.
Then there is an $n_1\in\bbN$ such that $\mu(A_{n_1})<1/2$,
hence $A_{n_1}\subset\{2,3,\dots\}$.
Again by $\mu(A_n)\to 0$, there is an $n_2\in\bbN$ such that $n_2>n_1$
and $\mu(A_{n_2})<1/2^2$,
hence $A_{n_2}\subset\{3,4,\dots\}$.
We continue in this fashion to obtain a subsequence $\{A_{n_i}\}_\seqi$
such that $A_{n_i}\subset\{i+1,i+2,\dots\}$ for every $\seqi$.
Then $\mu$ has property~(S) since it follows that
\[
\bigcap_{k=1}^\infty\bigcup_{i=k}^\infty A_{n_i}
\subset\bigcap_{k=1}^\infty\{k+1,k+2,\dots\}=\eset.
\]
\par
(2) For each $\seqn$, let $A_n:=\{1,2,\dots,n\}$.
Then $A_n\uparrow\bbN$ and $\mu(A_n)\to 1$, but $\mu(\bbN)=2$.
Hence $\mu$ is not continuous from below.
Next, for each $\seqn$, let $A_n:=\{n,n+1,\dots\}$. Then $A_n\downarrow\eset$,
but $\mu(A_n)>1$ for every $\seqn$. Hence $\mu$ is not order continuous.
\par 
(3) For each $\seqn$, let $A_n:=\{n\}$. Then, letting $k\to\infty$ shows
\[
\sup_{\seql}\mu\left(\bigcup_{n=k}^{k+l}A_n\right)=\sum_{i=k}^\infty\frac{1}{2^i}\to 0,
\]
but
\[
\mu\left(\bigcup_{n=k}^\infty A_n\right)=1+\sum_{i=k}^\infty\frac{1}{2^i}\to 1.
\]
Hence $\mu$ does not have property~(C).
Furthermore, $\mu(A_n)=1/2^n\to 0$,
but for any subsequence $\{A_{n_k}\}_\seqk$ of $\{A_n\}_\seqn$ it follows that
\[
\mu\left(\bigcup_{i=k}^\infty A_{n_i}\right)=\mu(\{n_k,n_{k+1},\dots\})>1
\]
for every $\seqk$. Therefore, $\mu$ does not have property~($\mbox{S}_1$).
\par
(4)\ Let $\ep>0$ and find  $n_0\in\bbN$ such that for any $n,l\in\bbN$, if $n\geq n_0$ then
$\sum_{i=n+1}^{n+l}1/2^i<\ep$. Then
\begin{align*}
\mu(\{|f_{n+l}-f_n|>\ep\})
&\leq\mu(A_{n+l}\setminus A_n)\\
&=\mu(\{n+1,\dots,n+l\})\\
&=\sum_{i=n+1}^{n+l}\frac{1}{2^i}<\ep,
\end{align*}
which means that $\{f_n\}_\seqn$ is Cauchy in $\mu$-measure.
\par
Next we prove that $\{f_n\}_\seqn$ does not converge in $\mu$-measure.
Suppose, contrary to our claim, that there is an $f\in\calF_0(X)$ such that $f_n\mconv f$.
If there were $n_0\in X$ such that $f(n_0)\ne 1$, then for any $n\geq n_0$
we would have $n_0\in\{|f_n-f|>\ep_0\}$,  where $\ep_0:=|1-f(n_0)|/2>0$, hence
\[
\mu(\{|f_n-f|>\ep_0\})\geq\frac{1}{2^{n_0}}>0,
\]
which contradicts our assumption. Hence $f(x)=1$ for every $x\in X$.
Therefore, letting $\ninfty$ shows
\[
\mu(\{|f_n-f|>1/2\})=\mu(\{n+1,n+2,\dots\})=1+\sum_{i=n+1}^\infty\frac{1}{2^i}\to 1,
\]
which contradicts the fact that $f_n\mconv f$.
\par
Finally, suppose that $\{f_n\}_\seqn$ has a subsequence $\{f_{n_k}\}_\seqk$ converging
in $\mu$-measure to a function $f\in\calF_0(X)$.
Since $\mu$ has the (p.g.p.)~and $\{f_n\}_\seqn$ is Cauchy in $\mu$-measure,
it follows from (2) of Theorem~\ref{base}
that $f_n\mconv f$, which contradicts the fact
that $\{f_n\}_\seqn$ does not converge in $\mu$-measure.
\par
(5)\ For any $\ep>0$ and $\seqn$, it follows that
\[
\mu(\{|f_n-1|>\ep\})\leq\mu(A_n^c)=\frac{1}{2^n},
\]
hence $f_n\mconv 1$. Since $\mu$ has the (p.g.p.),
$\{f_n\}_\seqn$ is Cauchy in $\mu$-measure by~(1) of Theorem~\ref{base}.
Suppose, contrary to our claim, that $\{f_n\}_\seqn$ has a subsequence $\{f_{n_k}\}_\seqk$
converging $\mu$-almost uniformly to a function $f\in\calF_0(X)$.
Since $f_{n_k}\mconv 1$ and since $\mu$ is null-continuous and has the (p.g.p.),
it follows from Lemma~\ref{au} that $f=1$ $\mu$-a.e.~and $f_{n_k}\to 1$ $\mu$-a.u.
Hence, there is an $E\in\calA$ such that $\mu(E)<1/2$ and
\begin{equation}\label{eqn:counter1}
\sup_{x\not\in E}|f_{n_k}(x)-1|\to 0.
\end{equation}
Then $E$ is at most a finite subset of $X$ since $\mu(E)<1/2$.
If $E=\eset$, then $n_k\not\in E$ for every $\seqk$, hence
\[
1\geq\sup_{x\not\in E}|f_{n_k}(x)-1|\geq |f_{n_k}(n_k)-1|=1,
\]
which implies that $\sup_{x\not\in E}|f_{n_k}(x)-1|=1$.
If $E$ is a nonempty finite subset of $X$, then there is a $k_0\in\bbN$ such that
$n_k\not\in E$ for every $k\geq k_0$, hence $\sup_{x\not\in E}|f_{n_k}(x)-1|=1$.
Both cases thus contradict~\eqref{eqn:counter1}.
\end{proof}
The following examples show that Theorems~\ref{comp} and~\ref{fund} are
sharpened versions of the corresponding results in~\cite{J-S-W-K-L-Y}.
We say that a nonadditive measure $\mu$ satisfies \emph{condition~($**$)}\/ if
for any $\ep>0$ there is a $\delta>0$ such that condition~($*$) in Proposition~\ref{ex-C-pro}
holds for every sequence $\{E_n\}_\seqn\subset\calA$.
\begin{example}\label{counter2}
Let $\nu\in\calM(X)$.
Assume that $\nu$ is continuous from below and subadditive.
Define the nonadditive measure $\mu\colon\calA\to [0,\infty]$ by
\[
\mu(A):=\begin{cases}
\nu(A) & \mbox{if }\nu(A)<1,\\
1+\nu(A) & \mbox{if }\nu(A)\geq 1
\end{cases}
\]
for every $A\in\calA$.
Then $\mu$ has property~(C) and the (p.g.p.).
In addition, it is null-additive and null-continuous, but not continuous from below.
Moreover, in the case where $\nu$ is the Lebesgue measure on $\bbR$,
we see that $\mu$ does not satisfy condition ($**$).
In fact, for any $\delta>0$, let $\delta_0:=(1\land\delta)/2>0$
and $E_n:=[(n-1)\delta_0,n\delta_0]$ for every $\seqn$.
Then it follows that $\sup_\seqn\mu(E_n)<\delta$,
but $\mu(\bigcup_{n=1}^\infty E_n)=\infty$.
This means that condition~($**$) is strictly stronger than
the conjunction of property~(C) and the (p.g.p.).
\end{example}
\begin{example}\label{counter3}
Let $X:=\bbN$ and $\calA:=2^\bbN$.
Define the nonadditive measure $\mu\colon\calA\to [0,1]$ by
\[
\mu(A):=\begin{cases}
1 & \mbox{if }A=X,\\
1/2 & \mbox{if }\eset\subsetneq A\subsetneq X,\\
0 & \mbox{if }A=\eset
\end{cases}
\]
for every $A\in\calA$.
Then $\mu$ has property~(C) and the (p.g.p.).
In addition, it is null-additive and null-continuous, but not continuous from below.
In more detail, for any sequence $\{E_n\}_\seqn\subset\calA$ we have
\[
\mu\left(\bigcup_{n=1}^\infty E_n\right)\leq 2\sup_\seqn\mu(E_n),
\]
so that $\mu$ satisfies condition ($**$) by Proposition~\ref{ex-C-pro}.
This shows that the continuity from below does not follow from condition~($**$).
\end{example}

%
% Section 4
%
\section{The necessity of property (C)}\label{necessity}
In this section, we discuss whether property~(C) is a necessary condition
for the Cauchy criterion to hold in the case where $X$ is countable.
To this end we prepare the following lemmas.
\begin{lemma}\label{countable}
Let $\mu\in\calM(X)$. Assume that $\mu$ is weakly null-additive.
If $X$ is countable, then the following assertions are equivalent.
\begin{enumerate}
\item[\us{(i)}] $\mu$ has property~(S).
\item[\us{(ii)}] $\mu$ is null-continuous.
\item[\us{(iii)}] For any sequence $\{E_n\}_\seqn\subset\calA$,
if $\mu(E_n)\to 0$ then $\mu\left(\bigcap_{k=1}^\infty\bigcup_{n=k}^\infty E_n\right)=0$.
\end{enumerate}
\end{lemma}
\begin{proof}
(i)$\Leftrightarrow$(ii)\ It follows from~\cite[Proposition~3.2]{U-M}.
\par
(ii)$\Rightarrow$(iii)\ We may assume that $\bigcup_{n=1}^\infty E_n\ne\eset$.
For each $x\in\bigcup_{n=1}^\infty E_n$,
let $\calE_x$ be the collection of all sets $E_n$ containing $x$,
and let $E(x)$ be the intersection of all sets in $\calE_x$.
Since $D:=\bigcap_{k=1}^\infty\bigcup_{n=k}^\infty E_n$ is at most countable,
it may be expressed by $\{a_m\colon m\in T\}$, where
$T=\bbN$ or $T=\{1,2,\dots,N\}$ for some $N\in\bbN$.
Then $D=\bigcup_{m\in T}E(a_m)$.
\par
Fix $\seqm$. Since $a_m\in\bigcup_{n=1}^\infty E_n$, there is an $n_1^{(m)}\in\bbN$
such that $a_m\in E_{n_1^{(m)}}$.
Then $a_m\in\bigcup_{n=n_1^{(m)}+1}^\infty E_n$,
hence there is an $n_2^{(m)}\in\bbN$ such that
$n_2^{(m)}>n_1^{(m)}$ and $a_m\in\ E_{n_2^{(m)}}$.
We continue in this fashion to obtain an increasing sequence $\{n_j^{(m)}\}_\seqj$ such that
$E(a_m)\subset E_{n_j^{(m)}}$ for every $\seqj$.
Hence $\mu(E(a_m))=0$ for every $m\in T$ since $\mu(E_n)\to 0$.
Therefore, $T$ being at most countable,  $\mu(D)=0$
by the weak null-additivity and null-continuity of $\mu$.
\par
(iii)$\Rightarrow$(i)\ It is obvious.
\end{proof}
\begin{remark}
Assertion (iii) of Lemma~\ref{countable} implies the weak null-additivity of $\mu$.
Indeed, take $A,B\in\calA$ and assume that $\mu(A)=\mu(B)=0$.
For each $\seqn$, let
\[
E_n:=\begin{cases}
A & \mbox{if $n$ is even},\\
B & \mbox{if $n$ is odd}.
\end{cases}
\]
Then $\mu(E_n)=0$ for every $\seqn$.
Hence, assertion (iii) yields
$\mu(A\cup B)=\mu\left(\bigcap_{k=1}^\infty\bigcup_{n=k}^\infty E_n\right)=0$,
which implies the weak null-additivity of $\mu$.
\end{remark}
\begin{lemma}\label{null-conti}
Let $\mu\in\calM(X)$. If any Cauchy in $\mu$-measure sequence $\{f_n\}_\seqn$ has
a subsequence converging $\mu$-almost everywhere, then $\mu$ is null-continuous.
\end{lemma}
\begin{proof}
Take a nondecreasing sequence $\{N_n\}_\seqn\subset\calA$ and assume that $\mu(N_n)=0$
for every $\seqn$. Let $N:=\bigcup_{n=1}^\infty N_n$ and assume that $N\ne\eset$
without loss of generality.
For each $\seqn$, let $f_n:=n\chi_{N_n}$.
Then $\{f_n\}_\seqn$ is a Cauchy in $\mu$-measure sequence in $\calF_0(X)$
such that $f_n(x)\to\infty$ for every $x\in N$.
Hence, there exist a function $f\in\calF_0(X)$
and a subsequence $\{f_{n_k}\}_\seqk$ of $\{f_n\}_\seqn$ such that $f_{n_k}\to f$ $\mu$-a.e.,
so that there is an $N_0\in\calA$ such that $\mu(N_0)=0$ and $f_{n_k}(x)\to f(x)$
for every $x\not\in N_0$.
Since $f_n(x)\to\infty$ for every $x\in N$ and $f$ is real-valued, we have $N\subset N_0$,
hence $\mu(N)=0$. Thus $\mu$ is null-continuous.
\end{proof}
\begin{theorem}\label{NSC0}
Let $\mu\in\calM(X)$. Assume that $\mu$ has the (p.g.p.).
If $X$ is countable, then the following assertions are equivalent.
\begin{enumerate}
\item[\us{(i)}] $\mu$ has property~($\mbox{C}_0$).
\item[\us{(ii)}] Any Cauchy in $\mu$-measure sequence $\{f_n\}_\seqn\subset\calF_0(X)$
has a subsequence converging $\mu$-almost everywhere.
\end{enumerate}
\end{theorem}
\begin{proof}
(i)$\Rightarrow$(ii)\ It follows from Corollary~\ref{AE}.
\par
(ii)$\Rightarrow$(i)\ Take $\{E_n\}_\seqn\subset\calA$ and assume that
\[
\sup_\seql\mu\left(\bigcup_{n=k}^{k+l}E_n\right)\to 0.
\]
Then $\mu(E_n)\to 0$, so that Lemma~\ref{countable} yields
$\mu\left(\bigcap_{k=1}^\infty\bigcup_{n=k}^\infty E_n\right)=0$
since $\mu$ is null-continuous by Lemma~\ref{null-conti} and weakly null-additive.
Hence $\mu$ has property~($\mbox{C}_0$).
\end{proof}
\begin{theorem}\label{NSC}
Let $\mu\in\calM(X)$. Assume that $\mu$ has the (p.g.p.).
If $X$ is countable, then the following assertions are equivalent.
\begin{enumerate}
\item[\us{(i)}] $\mu$ has property~(C).
\item[\us{(ii)}] Any Cauchy in $\mu$-measure sequence $\{f_n\}_\seqn\subset\calF_0(X)$
has a subsequence converging $\mu$-almost uniformly.
\item[\us{(iii)}] $\mu$ is null-continuous and any Cauchy in $\mu$-measure
sequence $\{f_n\}_\seqn\subset\calF_0(X)$ has a subsequence converging in $\mu$-measure.
\item[\us{(iv)}] $\mu$ is null-continuous and any Cauchy in $\mu$-measure
sequence $\{f_n\}_\seqn\subset\calF_0(X)$ converges in $\mu$-measure.
\end{enumerate}
\end{theorem}
\begin{proof}
(i)$\Rightarrow$(ii)\ It follows from Theorem~\ref{comp}.
\par
(ii)$\Rightarrow$(iii)\ It follows from Lemma~\ref{null-conti}.
\par
(iii)$\Rightarrow$(iv)\ It follows from (2) of Theorem~\ref{base}.
\par
(iv)$\Rightarrow$(i)\ Take $\{E_n\}_\seqn\subset\calA$ and assume that
\begin{equation}\label{eqn:NSC1}
\sup_\seql\mu\left(\bigcup_{n=k}^{k+l}E_n\right)\to 0.
\end{equation}
Let $D:=\bigcap_{k=1}^\infty\bigcup_{n=k}^\infty E_n$.
Since $\mu(E_n)\to 0$ by~\eqref{eqn:NSC1}
and since $\mu$ is weakly null-additive and null-continuous,
$\mu(D)=0$ follows from Lemma~\ref{countable}.
For each $\seqn$, let
$A_n:=\left(\bigcup_{i=n}^\infty E_i\right)\setminus D$ and $f_n:=\chi_{A_n}$.
Then $A_n\downarrow\eset$, hence $f_n(x)\to 0$ for every $x\in X$.
\par
We first show that $\{f_n\}_\seqn$ is Cauchy in $\mu$-measure.
To this end, observe that for any $k,l\in\bbN$, if $x\not\in\bigcup_{n=k}^{k+l}E_n$,
then $f_{k+l}(x)=f_k(x)$.
In fact, if $x\not\in A_k$, then $x\not\in A_{k+l}$, hence $f_{k+l}(x)=f_k(x)=0$.
If $x\in A_k$, then the set inclusion
\[
A_k=\left\{\left(\bigcup_{n=k}^{k+l-1}E_n\right)\setminus D\right\}
\cup\left\{\left(\bigcup_{n=k+l}^\infty E_n\right)\setminus D\right\}
\subset\left(\bigcup_{n=k}^{k+l} E_n\right)\cup A_{k+l}
\]
implies that $x\in A_{k+l}$ since $x\not\in\bigcup_{n=k}^{k+l}E_n$.
Hence $f_{k+l}(x)=f_k(x)=1$.
Therefore, for given $\ep>0$ and $\delta>0$, by~\eqref{eqn:NSC1} there is a $k_0\in\bbN$
such that for any $k,l\in\bbN$, if $k\geq k_0$ then
\[
\mu(\{|f_{k+l}-f_k|>\delta\})\leq\mu\left(\bigcup_{n=k}^{k+l}E_n\right)<\ep.
\]
Thus $\{f_n\}_\seqn$ is Cauchy in $\mu$-measure.
\par
Now, by assertion (iv) there is a function $f\in\calF_0(X)$ such that $f_n\mconv f$.
Since $\mu$ has property~(S) by Lemma~\ref{countable}, the Riesz theorem shows that
$\{f_n\}_\seqn$ has a subsequence $\{f_{n_k}\}_\seqk$ such that $f_{n_k}\to f$ $\mu$-a.e.
Hence $\mu(\{f\ne 0\})=0$ since $f_n(x)\to 0$ for every $x\in X$. Thus
\begin{equation}\label{eqn:NSC2}
\mu(\{|f_n-f|>1/2\}\cup\{f\neq 0\})\to 0
\end{equation}
since $f_n\mconv f$ and $\mu$ has the (p.g.p.).
Observe that
\[
A_n=\{|f_n|>1/2\}\subset\{|f_n-f|>1/2\}\cup\{f\ne 0\}
\]
for every $\seqn$. From this observation and~\eqref{eqn:NSC2} it follows that $\mu(A_n)\to 0$.
Since $\mu(D)=0$ and $\bigcup_{n=k}^\infty E_n\subset A_k\cup D$ for every $\seqk$,
the (p.g.p.)~of $\mu$ implies that $\mu\left(\bigcup_{n=k}^\infty E_n\right)\to 0$.
Therefore, $\mu$ has property~(C).
\end{proof}
\section{The space of the measurable functions}\label{MFS}
Prenorms on the space of all measurable functions have already been introduced
in~\cite{H-O,Kawabe2021,Li,O-Z,W-R-W} for nonadditive measures.
In~\cite{W-R-W}, the prenorm is given by the Sugeno integral,
while in~\cite{O-Z} by the Choquet integral.
The prenorm in~\cite{Li} is simpler than ours (Definition~\ref{DS}) for a description.
However, the prenorms in~\cite{Li,O-Z} may take the infinite value unless $\mu$ is finite,
and the prenorm in~\cite{W-R-W} may be incompatible with our purpose since the Lorentz space
is defined by the Choquet integral and the Lorentz space of weak type
is defined by the Shilkret integral.
We thus introduce a prenorm on the real linear space $\calF_0(X)$
with the property that it always takes finite values and is independent of the specific integrals.
To this end, define the function $\varphi\colon [0,\infty]\to [0,\pi/2]$ by
\[
\varphi(t):=\begin{cases}
\arctan t & \mbox{if }t\ne\infty\\
\pi/2 & \mbox{if }t=\infty.
\end{cases}
\]
Then $\varphi$ is a continuous and increasing function with the property that $\varphi(t)=0$
if and only if $t=0$. 
It satisfies $\varphi(t)\leq t$ and
$\varphi(s+t)\leq\varphi(s)+\varphi(t)$ for every $s,t\in [0,\infty]$.
In addition, given a constant $M>0$, it follows that $\varphi(Mt)\leq\max\{1,M\}\varphi(t)$
for every $t\in [0,\infty]$.
\begin{remark}
We may use any functions having the properties above
instead of the inverse tangent.
For instance, the function $\varphi\colon [0,\infty]\to [0,1]$ defined by
\[
\varphi(t):=\begin{cases}
t/(1+t) & \mbox{if }0\leq t<\infty,\\
1 & \mbox{if }t=\infty
\end{cases}
\]
has all the properties above.
\end{remark}
\begin{definition}[\cite{D-S,H-O,Kawabe2021,Rao}]\label{DS}
Let $\mu\in\calM(X)$.
Define the prenorm $\dsn{\cdot}$ on $\calF_0(X)$ by
\[
\dsn{f}:=\inf_{c>0}\varphi(c+\mu(\{|f|>c\}))
\]
for every $f\in\calF_0(X)$.
If the measure $\mu$ is needed to specify, then $\dsn{f}$ is written as $\dsnm{f}{\mu}$.
\end{definition}
In what follows, when the space $\calF_0(X)$ is considered
together with the prenorm $\dsn{\cdot}$,
it is written as $\calL^0(\mu)$ since the prenorm depends
on the nonadditive measure $\mu$.
\par
The prenorm $\dsn{\cdot}$ is truncated subhomogeneous, that is,
\[
\dsn{cf}\leq\max\{1,|c|\}\dsn{f}
\]
for every $f\in\calL^0(\mu)$ and $c\in\bbR$,
but not satisfy the triangle inequality in general.
Furthermore, for any $f\in\calL^0(\mu)$ it follows that
$\dsn{f}=0$ if and only if $\mu(\{|f|>c\})=0$ for every $c>0$;
they are equivalent to $\mu(\{|f|>0\})=0$ if $\mu$ is null-continuous.
The same argument as that in ordinary measure theory shows that
for any sequence $\{f_n\}_\seqn\subset\calL^0(\mu)$ and $f\in\calL^0(\mu)$,
convergence with respect to $\dsn{\cdot}$ is equivalent to convergence in $\mu$-measure, that is,
it follows that $\dsn{f_n-f}\to 0$ if and only if $f_n\mconv f$.
In the same way, $\{f_n\}_\seqn$ is Cauchy with respect to $\dsn{\cdot}$
if and only if $\{f_n\}_\seqn$ is Cauchy in $\mu$-measure.
The prenorm $\dsn{\cdot}$ also has the following properties, some of which are related to
the characteristic of nonadditive measures.
\begin{proposition}\label{PDS}
Let $\mu\in\calM(X)$.
\begin{enumerate}
\item[\us{(1)}] For any $A\in\calA$ and $c\in\bbR$
it follows that $\dsn{c\chi_A}=\min\{\varphi(\mu(A)),\varphi(|c|)\}$.
\item[\us{(2)}] For any $f,g\in\calL^0(\mu)$, if $|f|\leq |g|$,
then $\dsn{f}\leq\dsn{g}$. 
\item[\us{(3)}] $\mu$ is weakly null-additive if and only if $\dsn{\cdot}$ is weakly null-additive.
\item[\us{(4)}] $\mu$ is null-additive if and only if $\dsn{\cdot}$ is null-additive.
\item[\us{(5)}] If $\mu$ is $K$-relaxed subadditive for some $K\geq 1$,
then $\dsn{\cdot}$ satisfies the $K$-relaxed triangle inequality.
In particular, if $\mu$ is subadditive, then $\dsn{\cdot}$ satisfies the triangle inequality.
\item[\us{(6)}] $\mu$ is null-additive if and only if it follows that $\dsn{f}=\dsn{g}$ whenever
$f,g\in\calL^0(\mu)$ and  $f\sim g$.
\end{enumerate}
\end{proposition}
\begin{proof}
Assertions (1)--(5) are easy to prove.
Some of their proofs can be found in~\cite[Proposition~3.2]{Kawabe2021}.
See also~\cite{H-O}.
\par
(6)\ The ``only if'' part. Let$f,g\in\calL^0(\mu)$ and assume that $f\sim g$.
Then $\mu(\{|f-g|>c\})=0$ for every $c>0$, so that $\dsn{f-g}=0$.
Since $\mu$ is null-additive, so is $\dsn{\cdot}$ by (2), hence
$\dsn{f}=\dsn{g+(f-g)}=\dsn{g}$ follows.
\par
The ``if'' part. Take $A,B\in\calA$ and assume that $\mu(B)=0$.
For each $r>0$, let $f:=r\chi_{A\cup B}$ and $g:=r\chi_A$.
Then $\mu(\{|f-g|>c\})=0$ for every  $c>0$, hence $f\sim g$, and finally $\dsn{f}=\dsn{g}$.
It thus follows from (1) that
\[
\min\left\{\varphi(\mu(A\cup B),\varphi(r)\right\}=\dsn{f}=
\dsn{g}=\min\left\{\varphi(\mu(A),\varphi(r)\right\}
\]
for every $r>0$. Letting $r\uparrow\infty$ yields $\varphi(\mu(A\cup B))=\varphi(\mu(A))$,
which implies $\mu(A\cup B)=\mu(A)$ since $\varphi$ is injective.
Thus $\mu$ is null-additive.
\end{proof}
Next let us construct the quotient space of $\calL^0(\mu)$ by the equivalence relation
defined in Subsection~\ref{equiv}.
Let
\[
L^0(\mu):=\{[f]\colon f\in\calL^0(\mu)\}.
\]
Given an equivalence class $[f]\in L^0(\mu)$, define the prenorm on $L^0(\mu)$ by
$\dsn{[f]}:=\dsn{f}$, which is well-defined if $\mu$ is null-additive
by (6) of Proposition~\ref{PDS}.
This prenorm has the same properties as the prenorm on $\calL^0(\mu)$
and separates points of $L^0(\mu)$, that is, for any $[f]\in L^0(\mu)$,
if $\dsn{[f]}=0$ then $[f]=0$.
\par
Now, the Cauchy criterion with respect to $\dsn{\cdot}$ can be deduced
from Theorem~\ref{fund}.
\begin{theorem}\label{calL0}
Let $\mu\in\calM(X)$. Assume that $\mu$ has property~(C) and the (p.g.p.).
\begin{enumerate}
\item[\us{(1)}] A sequence $\{f_n\}_\seqn\subset\calL^0(\mu)$ is Cauchy
if and only if it converges.
\item[\us{(2)}] Additionally, assume that $\mu$ is null-additive.
A sequence $\{[f_n]\}_\seqn\subset L^0(\mu)$ is Cauchy if and only if it converges.
\end{enumerate}
\end{theorem}
\begin{corollary}\label{L0}
Let $\mu\in\calM(X)$. Assume that $\mu$ has property~(C) and the (p.g.p.).
Then $\calL^0(\mu)$ is complete.
If $\mu$ is additionally assumed to be null-additive, then $L^0(\mu)$ is complete.
\end{corollary}
\begin{remark}
Every nonadditive measure that is continuous from below has property (C).
Furthermore, by Examples~\ref{counter2} and~\ref{counter3} there are nonadditive
measures having property~(C) and the (p.g.p.), but not continuous from below.
Consequently, Theorem~\ref{calL0} and Corollary~\ref{L0} are sharpened versions
of those given in~\cite{Kawabe2021}.
\end{remark}
\begin{example}
Let $X:=\bbN$ and $\calA:=2^\bbN$.
Let $\mu$ be the nonadditive measure given in Proposition~\ref{counter1}.
Then $\mu$ has the (p.g.p.)~but not property~(C).
For each $\seqn$, let $A_n:=\{1,2,\dots,n\}$ and $f_n:=\chi_{A_n}$.
Then by (4) of Proposition~\ref{counter1} the sequence
$\{f_n\}_\seqn\subset\calL^0(\mu)$ is Cauchy but does not converge.
Hence $\calL^0(\mu)$ is not complete.
This fact shows that property~(C) cannot be dropped
in Theorem~\ref{calL0} and Corollary~\ref{L0}.
\end{example}
The results for dense subsets and the separability of $\calL^0(\mu)$ and $L^0(\mu)$
immediately follow from~\cite[Theorems~7.2 and~7.7]{Kawabe2021}.
Recall that a nonadditive measure $\mu$ has a \emph{countable basis}\/ if
there is a countable subset
$\calD$ of $\calA$, which is called a \emph{countable basis} of $\mu$,
such that for any $A\in\calA$ and $\ep>0$
there is a $D\in\calD$ for which $\mu(A\triangle D)<\ep$.
Recall that $S(X)=\{[h]\colon h\in\calS(X)\}$.
\begin{theorem}
Let $\mu\in\calM(X)$. Assume that $\mu$ is order continuous.
Then $\calS(X)$ is dense in $\calL^0(\mu)$.
If $\mu$ is additionally assumed to be null-additive, then $S(X)$ is dense in $L^0(\mu)$.
\end{theorem} 
\begin{theorem}\label{L0sep}
Let $\mu\in\calM(X)$. Assume that $\mu$ is order continuous and has the (p.g.p.).
Assume that $\mu$ has a countable basis.
Then there is a countable subset $\calE$ of $\calL^0(\mu)$
such that for any $f\in\calL^0(\mu)$ and $\ep>0$ one can find $h\in\calE$ for which $\dsn{f-h}<\ep$.
Hence $\calL^0(\mu)$ is separable.
If $\mu$ is additionally assumed to be null-additive, then $L^0(\mu)$ is separable.
\end{theorem}
\section{The Choquet-Lorentz space}\label{CLspace}
In this section, a type of Lorentz space is defined by using the Choquet integral
instead of the Lebesgue integral, and its completeness and separability are discussed.
\begin{definition}\label{C-L}
Let $\mu\in\calM(X)$. Define the function $\pqn{\cdot}\colon\calF_0(X)\to [0,\infty]$ by
\[
\pqn{f}:=\left(\frac{p}{q}\right)^{1/q}\Ch(\mupq,|f|^q)^{1/q}
\]
for every $f\in\calF_0(X)$ and let
\[
\calLpq(\mu):=\{f\in\calF_0(X)\colon\pqn{f}<\infty\}.
\]
If the measure $\mu$ is needed to specify, then $\pqn{f}$ is written as $\pqnm{f}{\mu}$.
The space $\calLpq(\mu)$ is
called the \textit{Choquet-Lorentz space}\/ and the prenorm $\pqn{\cdot}$
on $\calLpq(\mu)$ is called the \textit{Choquet-Lorentz prenorm}.
In particular, if $p=q$ then the space $\calL^{q,q}(\mu)$ is
called the \textit{Choquet $\calL^q$ space}
and $\|\cdot\|_{q,q}$ is called the \textit{Choquet $\calL^q$ prenorm}.
They are denoted by $\calL^q(\mu)$ and $\qn{\cdot}$, respectively.
In other words,
\[
\qn{f}=\Ch(\mu,|f|^q)^{1/q}
\]
for every $f\in\calF_0(X)$ and
\[
\calLq(\mu)=\{f\in\calF_0(X)\colon\qn{f}<\infty\}.
\]
\end{definition}
\par
If the prenorm $\pqn{\cdot}$ on $\calLpq(\mu)$ is a quasi-seminorm,
then $\calLpq(\mu)$ is a real linear subspace of $\calF_0(X)$, but this is not the case in general. 
When $\mu$ is $\sigma$-additive, the Choquet-Lorentz space coincides with the Lorentz space
and the Choquet $\calL^q$ space coincides with the Lebesgue space
of all $q$-th order integrable functions,
both of which are defined by the abstract Lebesgue integral~\cite[Theorem~6.6]{C-R}.
By definition it follows that
\[
\pqnm{f}{\mu}=\left(\frac{p}{q}\right)^{1/q}\qnm{f}{\mupq}
\]
for every $f\in\calL^{p,q}(\mu)$ and that $\calL^{p,q}(\mu)=\calL^q(\mupq)$.
This observation is important and shows that some properties
of the Choquet Lorentz space may be deduced from those of the Choquet $\calL^q$ space;
see the proofs of Theorems~\ref{L-comp1}, \ref{L-dense}, and~\ref{L-sep}.
The prenorm $\pqn{\cdot}$ does not satisfy the triangle inequality in general
even if $\mu$ is $\sigma$-additive~\cite[Theorem~6.5]{C-R}.
\begin{proposition}\label{pqnorm}
Let $\mu\in\calM(X)$. Let $0<p<\infty$ and $0<q<\infty$.
\begin{enumerate}
\item[\us{(1)}] For any $A\in\calA$ it follows that
\[
\pqn{\chi_A}=\left(\dfrac{p}{q}\right)^{1/q}\mu(A)^{1/p}.
\]
\item[\us{(2)}] For any $f\in\calLpq(\mu)$ it follows that $\pqn{f}=0$
if and only if $\mu(\{|f|>c\})=0$ for every $c>0$;
they are equivalent to $\mu(\{|f|>0\})=0$ if $\mu$ null-continuous.
\item[\us{(3)}] For any $f\in\calLpq(\mu)$ and $c\in\bbR$
it follows that $\pqn{cf}=|c|\pqn{f}$.
Hence $\pqn{\cdot}$ is homogeneous.
\item[\us{(4)}] For any $f\in\calLpq(\mu)$ and $c>0$ it follows that
\[
\mu(\{|f|>c\})\leq\frac{1}{c^p}\left(\frac{q}{p}\right)^{p/q}\pqn{f}^p.
\]
\item[\us{(5)}] For any $f,g\in\calLpq(\mu)$, if $|f|\leq |g|$ then $\pqn{f}\leq\pqn{g}$.
\item[\us{(6)}] $\mu$ is weakly null-additive if and only if $\pqn{\cdot}$ is weakly null-additive.
\item[\us{(7)}] $\mu$ is null-additive if and only if $\pqn{\cdot}$ is null-additive.
\item[\us{(8)}] If $\mu$ is $K$-relaxed subadditive for some $K\geq 1$,
then $\pqn{\cdot}$ satisfies
the $2^{1+\frac{1}{p}+\frac{1}{q}}K^\frac{1}{p}$-relaxed triangle inequality.
Hence $\pqn{\cdot}$ is a quasi-seminorm.
\item[\us{(9)}] $\mu$ is null-additive if and only if it follows that $\pqn{f}=\pqn{g}$
whenever $f,g\in\calLpq(\mu)$ and $f\sim g$.
\end{enumerate}
\end{proposition}
\begin{proof}
Assertions (1)--(5) are easy to prove
and assertions (6) and (7) can be derived
in the same manner as~\cite[Proposition~3.2]{Kawabe2021}.
\par
(8)\ Let $f,g\in\calLpq(\mu)$. For any $t>0$, we have
\[
\{|f+g|^q>t\}\subset\{2^q|f|^q>t\}\cup\{2^q|g|^q>t\},
\]
hence
\begin{equation}\label{eqn:pqnorm1}
\mu(\{|f+g|^q>t\})\leq K\bigl\{\mu(\{2^q|f|^q>t\})+\mu(\{2^q|g|^q>t\})\bigr\}
\end{equation}
by the $K$-relaxed subadditivity of $\mu$. It thus follows that
\begin{align*}
&\hspace*{-7mm}\left(\frac{q}{p}\right)^{1/q}\pqn{f+g}\\
&=\left(\int_0^\infty\mu(\{|f+g|^q>t\})^{q/p}dt\right)^{1/q}\\
&\leq K^\frac{1}{p}\left(\int_0^\infty\bigl\{\mu(\{2^q|f|^q>t\})
+\mu(\{2^q|g|^q>t\})\bigr\}^{q/p}dt\right)^{1/q}\\
&\leq 2^\frac{1}{p}K^\frac{1}{p}\left(\int_0^\infty\mu(\{2^q|f|^q>t\})^{q/p}dt
+\int_0^\infty\mu(\{2^q|g|^q>t\})^{q/p}dt\right)^{1/q}\\
&=2^\frac{1}{p}K^\frac{1}{p}\left(2^q\Ch(\mupq,|f|^q)+2^q
\Ch(\mupq,|g|^q)\right)^{1/q}\\
&\leq 2^\frac{1}{p}K^\frac{1}{p}2^{1+\frac{1}{q}}\left(\Ch(\mupq,|f|^q)^{1/q}
+\Ch(\mupq,|g|^q)^{1/q}\right)\\
&=2^{1+\frac{1}{p}+\frac{1}{q}}K^\frac{1}{p}
\left(\frac{q}{p}\right)^{1/q}\left(\pqn{f}+\pqn{g}\right),
\end{align*}
where the first inequality is due to~\eqref{eqn:pqnorm1} and the second and third inequalities
are due to the elementary inequality
\begin{equation}\label{elementary}
|a+b|^r\leq 2^r(|a|^r+|b|^r)
\end{equation}
that holds for every $a,b\in\bbR$ and $0<r<\infty$.
\par
(9)\ It can be proved in the same manner as (6) of Proposition~\ref{PDS}.
\end{proof}
The quotient space
\[
L^{p,q}(\mu):=\{[f]\colon f\in\calLpq(\mu)\}
\]
is defined by the equivalence relation stated in Subsection~\ref{equiv}.
Given an equivalence class $[f]\in L^{p,q}(\mu)$,
define the prenorm on $L^{p,q}(\mu)$ by $\pqn{[f]}:=\pqn{f}$,
which is well-defined if $\mu$ is null-additive by (9) of Proposition~\ref{pqnorm}.  
This prenorm has the same properties as the prenorm on $\calLpq(\mu)$ and 
separates points of $L^{p,q}(\mu)$, that is, for any $[f]\in L^{p,q}(\mu)$,
if $\pqn{[f]}=0$ then $[f]=0$.
\par
To show the completeness of the spaces $\calL^{p,q}(\mu)$ and $L^{p,q}(\mu)$
some suitable convergence theorems of
the Choquet integral are needed.
Recall that every nonadditive measure that is monotone autocontinuous from below is null-additive.
\begin{proposition}\label{Ch-conv}
Let $\mu\in\calM(X)$. Let $0<q<\infty$.
The following assertions are equivalent.
\begin{enumerate}
\item[\us{(i)}] $\mu$ is monotone autocontinuous from below.
\item[\us{(ii)}] For any $\{f_n\}_\seqn\subset\calF_0(X)$ and $f\in\calF_0(X)$,
if they satisfy
\begin{enumerate}
\item[\us{(a)}] for each $\seqn$ there is an $N_n\in\calA$ such that $\mu(N_n)=0$
and $f_n(x)\leq f_{n+1}(x)\leq f(x)$ for every $x\not\in N_n$,
\item[\us{(b)}] $f_n\mconv f$,
\end{enumerate}
then $\mu(\{f_n>t\})\uparrow\mu(\{f>t\})$ for every continuity\
point $t$ of the function $\mu(\{f>t\})$.
\item[\us{(iii)}] The Choquet $q$-th monotone nondecreasing
almost uniform convergence theorem holds
for $\mu$, that is, for any $\{f_n\}_\seqn\subset\calF_0^+(X)$ and $f\in\calF_0^+(X)$, 
if they satisfy condition \us{(a)}\/ and
\begin{enumerate}
\item[\us{(c)}] $f_n\to f$ $\mu$-a.u.,
\end{enumerate}
then $\Ch(\mu,f_n^q)\uparrow\Ch(\mu,f^q)$.
\end{enumerate}
\end{proposition}
\begin{proof}
In this proof, for each $t\geq 0$ and $\seqn$, let $\varphi_n(t):=\mu(\{f_n>t\})$
and $\varphi(t):=\mu(\{f>t\})$.
\par
(i)$\Rightarrow$(ii)\ Let $t_0\in\bbR$ be a continuity point of $\varphi$.
The null-additivity of $\mu$ and condition (a) imply that
\[
\varphi_n(t)\leq\varphi_{n+1}(t)\leq\varphi(t)
\]
for every $t\in\bbR$ and $\seqn$, which yields $\sup_\seqn\varphi_n(t_0)\leq\varphi(t_0)$.
It thus suffices to show
\begin{equation}\label{eqn:Ch-conv1}
\varphi(t_0)\leq\sup_\seqn\varphi_n(t_0).
\end{equation}
To see this, fix $\ep>0$ and let $A:=\{f>t_0+\ep\}$ and $B_n:=\{|f_n-f|>\ep\}$
for every $\seqn$.
Then $\mu(B_n)\to 0$ by conditions (b).
By condition~(a) and the null-additivity of $\mu$, one can find a nondecreasing sequence $\{N_n\}_\seqn$
of $\mu$-null sets such that $f_n(x)\leq f_{n+1}(x)\leq f(x)$ for every $x\not\in N_n$.
Then it is easy to verify that $\{B_n\setminus N_n\}_\seqn$ is nonincreasing and
$A\setminus (B_n\setminus N_n)\subset\{f_n>t_0\}\cup N_n$, so that
\begin{equation}\label{eqn:Ch-conv2}
\mu(A\setminus (B_n\setminus N_n))\leq\mu(\{f_n>t_0\})
\end{equation}
for every $\seqn$.
Furthermore, the monotone autocontinuity of $\mu$ from below yields
\begin{equation}\label{eqn:Ch-conv3}
\mu(A)=\sup_\seqn\mu(A\setminus (B_n\setminus N_n))
\end{equation}
since $\mu(B_n\setminus N_n)\to 0$.
Therefore it follows from~\eqref{eqn:Ch-conv2} and~\eqref{eqn:Ch-conv3} that
\[
\varphi(t_0+\ep)\leq\sup_\seqn\varphi_n(t_0),
\]
which implies~\eqref{eqn:Ch-conv1} since $t_0$ is a continuity point of $\varphi$.
\par
(ii)$\Rightarrow$(iii)\ Let $D$ be the set of all discontinuity points of $\varphi$.
Condition (c) always implies condition (b), so that
$\varphi_n(t)\uparrow\varphi(t)$ for every $t\not\in D$ by assertion~(ii).
Hence, we have
\[
\Ch(\mu,f_n^q)=\int_0^\infty qt^{q-1}\varphi_n(t)dt\uparrow
\int_0^\infty qt^{q-1}\varphi(t)dt=\Ch(\mu,f^q)
\]
by the Lebesgue monotone convergence theorem
and the transformation formula of the Choquet integral.
\par
(iii)$\Rightarrow$(i)\ Take $A,B_n\in\calA$ and assume that $\{B_n\}_\seqn$ is nonincreasing
and $\mu(B_n)\to 0$. For any $\seqn$, let $f_n:=\chi_{A\setminus B_n}$ and $f:=\chi_A$.
Then they satisfy conditions~(a) and~(c), hence by assertion~(iii) we have 
\[
\mu(A\setminus B_n)=\Ch(\mu,f_n^q)\to\Ch(\mu,f^q)=\mu(A).
\]
Therefore, $\mu$ is monotone autocontinuous from below.
\end{proof}
\begin{remark}
See~\cite[Theorem~3.5]{Rebille} for a primitive form of Proposition~\ref{Ch-conv}.
\end{remark}
\begin{proposition}\label{C-Fatou}
Let $\mu\in\calM(X)$. Let $0<q<\infty$.
The following assertions are equivalent.
\begin{enumerate}
\item[\us{(i)}] $\mu$ is monotone autocontinuous from below.
\item[\us{(ii)}] The Choquet $q$-th Fatou almost uniform convergence lemma
holds for $\mu$, that is,
for any $\{f_n\}_\seqn\subset\calF_0^+(X)$ and $f\in\calF_0^+(X)$,
if $f_n\to f$ $\mu$-a.u., then $\Ch(\mu,f^q)\leq\liminf_\ninfty\Ch(\mu,f_n^q)$.
\end{enumerate}
\end{proposition}
\begin{proof}
(i)$\Rightarrow$(ii)\ For each $\seqn$, let $g_n:=\inf_{k\geq n}f_k$.
Since $\{g_n\}_\seqn$ and $f$ satisfy conditions (a) and (c) of Proposition~\ref{Ch-conv},
it follows that
\[
\Ch(\mu,f^q)
=\lim_\ninfty\Ch(\mu,g_n^q)=\lim_\ninfty\Ch(\mu,\inf_{k\geq n}f_k^q)\leq\liminf_\ninfty\Ch(\mu,f_n^q).
\]
\par
(ii)$\Rightarrow$(i)\ Take $A,B_n\in\calA$ and assume that $\{B_n\}_\seqn$
is nonincreasing and
$\mu(B_n)\to 0$. For any $\seqn$, let $f_n:=\chi_{A\setminus B_n}$ and $f:=\chi_A$.
Then $f_n\to f$ $\mu$-a.u., so that assertion (ii) yields
\begin{align*}
\mu(A)
&=\Ch(\mu,f^q)\leq\liminf_\ninfty\Ch(\mu,f_n^q)\\
&=\liminf_\ninfty\mu(A\setminus B_n)\leq\limsup_\ninfty\mu(A\setminus B_n)\leq\mu(A).
\end{align*}
Therefore $\mu$ is monotone autocontinuous from below.
\end{proof}
\begin{theorem}\label{L-comp1}
Let $\mu\in\calM(X)$. Let $0<p<\infty$ and $0<q<\infty$.
Assume that $\mu$ is monotone autocontinuous from below and has property~(C) and the (p.g.p.).
Then $\calL^{p,q}(\mu)$ and $L^{p,q}(\mu)$ are quasi-complete.
\end{theorem}
\begin{proof}
We first prove the conclusion for $\calLq(\mu)$.
Let $\{f_n\}_\seqn\subset\calLq(\mu)$ be bounded and Cauchy.
By (4) of Proposition~\ref{pqnorm}, the sequence $\{f_n\}_\seqn$ is Cauchy in $\mu$-measure,
so that by Theorem~\ref{comp} one can find
a subsequence $\{f_{n_k}\}_\seqk$ of $\{f_n\}_\seqn$
and a function $f\in\calF_0(X)$ such that $f_{n_k}\to f$ $\mu$-a.u.
\par
Let $\ep>0$.
Since $\{f_n\}_\seqn$ is Cauchy,
there is an $n_0\in\bbN$ such that if $m,n\geq n_0$ then
\begin{equation}\label{eqn:L-comp1}
\Ch(\mu,|f_m-f_n|^q)<\ep.
\end{equation}
For any fixed $n$ with $n\geq n_0$, we have $|f_{n_k}-f_n|\to |f-f_n|$ $\mu$-a.u.,
hence it follows from the monotone autocontinuity of $\mu$ from below
and Proposition~\ref{C-Fatou} that
\begin{align}\label{eqn:L-comp2}
\Ch(\mu,|f-f_n|^q)
&\leq\liminf_\kinfty\Ch(\mu,|f_{n_k}-f_n|^q)\nonumber\\
&\leq\limsup_\kinfty\Ch(\mu,|f_{n_k}-f_n|^q)\nonumber\\
&\leq\sup_{k\geq l}\Ch(\mu,|f_{n_k}-f_n|^q) 
\end{align}
for every $\seql$.
Since $n_k\to\infty$, there is a $k_0\in\bbN$ such that $n_k\geq n_0$ for every $k\geq k_0$,
hence it follows from \eqref{eqn:L-comp1} and \eqref{eqn:L-comp2} that
\[
\Ch(\mu,|f-f_n|^q)\leq\sup_{k\geq k_0}\Ch(\mu,|f_{n_k}-f_n|^q)\leq\ep,
\]
which yields $\qn{f-f_n}\to 0$.
\par
Next we show that $f\in\calLq(\mu)$.
Since $\{f_n\}_\seqn$ is bounded and $|f_{n_k}|\to |f|$ $\mu$-a.u., Proposition~\ref{C-Fatou} implies
\[
\qn{f}^q=\Ch(\mu,|f|^q)\leq\liminf_\kinfty\Ch(\mu,|f_{n_k}|^q)\leq
\sup_\seqn\qn{f_n}^q<\infty.
\]
Hence $f\in\calLq(\mu)$. Thus $\calLq(\mu)$ is quasi-complete.
Since every nonadditive measure that is monotone autocontinuous from below is null-additive,
the quotient space $L^q(\mu)$ and the quotient prenorm $\qn{\cdot}$ are well-defined,
and it turns out that $L^q(\mu)$ is quasi-complete.
\par
Recall that 
\[
\pqnm{f}{\mu}=\left(\frac{p}{q}\right)^{1/q}\qnm{f}{\mupq}
\]
for every $f\in\calLpq(\mu)$ and that $\calLpq(\mu)=\calLq(\mupq)$.
The conclusion for $\calLpq(\mu)$ and $L^{p,q}(\mu)$ thus follows
since $\mupq$ is monotone autocontinuous from below
and has property~(C) and the (p.q.p.)~if and only if so is $\mu$.
\end{proof}
We now reach the completeness of the spaces $\calLpq(\mu)$ and $L^{p,q}(\mu)$.
\begin{corollary}\label{L-comp2}
Let $\mu\in\calM(X)$. Let $0<p<\infty$ and $0<q<\infty$.
Assume that $\mu$ is relaxed subadditive, monotone autocontinuous from below,
and has property~(C).
Then $\calLpq(\mu)$ is complete with respect to the quasi-seminorm $\pqn{\cdot}$
and $L^{p,q}(\mu)$ is complete with respect to the quasi-norm $\pqn{\cdot}$.
\end{corollary}
\begin{proof}
By assumption, $\mu$ is also null-additive and has the (p.g.p.).
Furthermore, by (8) of Proposition~\ref{pqnorm} the prenorm $\pqn{\cdot}$
is a quasi-seminorm on $\calLpq(\mu)$, hence a quasi-norm on $L^{p,q}(\mu)$.
Therefore, every Cauchy sequence is bounded with respect to $\pqn{\cdot}$.
The conclusion thus follows from Theorem~\ref{L-comp1}.
\end{proof}
\begin{example}
Let $X:=\bbN$ and $\calA:=2^\bbN$.
Let $\mu$ be the nonadditive measure given in Proposition~\ref{counter1}.
Then $\mu$ is subadditive, hence relaxed subadditive, monotone autocontinuous from below,
and has the (p.g.p.),
while it does not have property~(C).
For each $\seqn$, let $A_n:=\{1,2,\dots,n\}$ and $f_n:=\chi_{A_n}$.
Then it follows from (4) of Proposition~\ref{counter1}
that $\{f_n\}_\seqn$ does not converge in $\mu$-measure.
Let $0<p<\infty$ and $0<q<\infty$. By (1) of Proposition~\ref{pqnorm} we have
\[
\pqn{f_n}=\left(\frac{p}{q}\right)^{1/q}\left(\sum_{i=1}^n\frac{1}{2^i}\right)^{1/p}
\]
and
\[
\pqn{f_{n+l}-f_n}=\left(\frac{p}{q}\right)^{1/q}
\left(\sum_{i=n+1}^{n+l}\frac{1}{2^i}\right)^{1/p}
\]
for every $n,l\in\bbN$, hence the sequence
$\{f_n\}_\seqn\subset\calLpq(\mu)$ is bounded and Cauchy.
Suppose, contrary to our claim, that $\calLpq(\mu)$ is quasi-complete.
Then, $\{f_n\}_\seqn$ converges with respect to $\pqn{\cdot}$,
thus it converges in $\mu$-measure by (4) of Proposition~\ref{pqnorm},
which is impossible.
Therefore $\calLpq(\mu)$ and $L^{p,q}(\mu)$ are not quasi-complete.
This means that property~(C) cannot be dropped 
in Theorem~\ref{L-comp1} and Corollary~\ref{L-comp2}.
\end{example}
We turn to the separability of the Choquet-Lorentz space.
For the proof a suitable convergence theorem is needed this time too.
\begin{proposition}\label{monotone}
Let $\mu\in\calM(X)$. The following assertions are equivalent.
\begin{enumerate}
\item[\us{(i)}] $\mu$ is conditionally order continuous.
\item[\us{(ii)}] For any $\{f_n\}_\seqn\subset\calF_0^+(X)$, if it satisfies
\begin{enumerate}
\item[\us{(a)}] $f_{n+1}(x)\leq f_n(x)$ for every $x\in X$ and $\seqn$,
\item[\us{(b)}] $f_n(x)\to 0$ for every $x\in X$,
\item[\us{(c)}] $\Ch(\mu,f_1)<\infty$,
\end{enumerate}
then $\Ch(\mu,f_n)\to 0$.
\end{enumerate}
\end{proposition}
\begin{proof}
(i)$\Rightarrow$(ii)\ For any $t>0$, condition (c) yields
\[
\mu(\{f_1\geq t\})\leq\frac{\Ch(\mu,f_1)}{t}<\infty,
\]
hence conditions (a) and (b) imply that $\mu(\{f_n\geq t\})\downarrow 0$
since $\mu$ is conditionally order continuous.
Furthermore, by condition (c) we have
\[
\int_0^\infty\mu(\{f_1\geq t\})=\Ch(\mu,f_1)<\infty,
\]
so that
\[
\Ch(\mu,f_n)=\int_0^\infty\mu(\{f_n\geq t\})dt\to 0
\]
by the monotone convergence theorem for the Lebesgue integral.
\par
(ii)$\Rightarrow$(i)\ Take $A_n\in\calA$ and assume that $A_n\downarrow\eset$ and $\mu(A_1)<\infty$.
For each $\seqn$, let $f_n:=\chi_{A_n}$.
Then the sequence $\{f_n\}_\seqn$ satisfies conditions (a)--(c), so that $\mu(A_n)=\Ch(\mu,f_n)\to 0$.
Hence $\mu$ is conditionally order continuous.
\end{proof}
\begin{theorem}\label{L-dense}
Let $\mu\in\calM(X)$. Let $0<p<\infty$ and $0<q<\infty$.
Assume that $\mu$ is conditionally order continuous.
Then $\calS(X)\cap\calL^{p,q}(\mu)$ is dense in $\calL^{p,q}(\mu)$.
If $\mu$ is additionally assumed to be null-additive, then $S(X)\cap L^{p,q}(\mu)$ is dense in $L^{p,q}(\mu)$.
\end{theorem}
\begin{proof}
We first prove the conclusion for $\calLq(\mu)$.
Let $f\in\calLq(\mu)$ and take a sequence $\{h_n\}_\seqn\subset\calS(X)$ such that
$|h_n(x)|\leq|f(x)|$ for every $\seqn$ and that $h_n(x)\to f(x)$ for every $x\in X$.
For any $x\in X$ and $\seqn$, let
$g_n(x):=\sup_{k\geq n}|f(x)-h_k(x)|^q$.
Then the sequence $\{g_n\}_\seqn$ satisfies conditions (a)--(c) of Proposition~\ref{monotone},
hence $\Ch(\mu,g_n)\to 0$.
For any $\seqn$,
\begin{align*}
\liminf_\ninfty\Ch(\mu,|f-h_n|^q)
&\leq\limsup_\ninfty\Ch(\mu,|f-h_n|^q)\\
&\leq\sup_{k\geq n}\Ch(\mu,|f-h_k|^q)\\
&\leq\Ch(\mu,g_n),
\end{align*}
hence $\qn{f-h_n}=\Ch(\mu,|f-h_n|^q)^{1/q}\to 0$ since $\Ch(\mu,g_n)\to 0$.
The fact that $h_n\in\calS(X)\cap\calLq(\mu)$ follows since $|h_n|\leq |f|$ and $f\in\calLq(\mu)$.
It thus follows that $\calS(X)\cap\calLq(\mu)$ is dense in $\calLq(\mu)$.
If $\mu$ is additionally assumed to be null-additive, then the quotient spaces $S(X)$ and $L^q(\mu)$
and the prenorm $\qn{\cdot}$ are well-defined,
and it turns out that $S(X)\cap L^q(\mu)$ is dense in $L^q(\mu)$. 
\par
The conclusion for $\calLpq(\mu)$ and $L^{p,q}(\mu)$ thus follows
since $\mupq$ is conditionally order continuous and null-additive if and only if so is $\mu$.
\end{proof}
\begin{theorem}\label{L-sep}
Let $\mu\in\calM(X)$. Let $0<p<\infty$ and $0<q<\infty$.
Assume that $\mu$ is conditionally order continuous and relaxed subadditive.
Assume that $\mu$ has a countable basis of sets in $\calA$
with finite $\mu$-measure.
Then there is a countable subset $\calE$ of $\calLpq(\mu)$ such that
for any $f\in\calLpq(\mu)$ and $\ep>0$ there is a $\psi\in\calE$ such $\pqn{f-\psi}<\ep$.
Hence $\calLpq(\mu)$ is separable.
If $\mu$ is additionally assumed to be null-additive, then $L^{p,q}(\mu)$ is separable.
\end{theorem}
\begin{proof}
We first prove the conclusion for $\calLq(\mu)$.
To this end, observe that there is a constant $L>2$ such that
\begin{equation}\label{eqn:L-sep1}
\qn{f_1+f_2+\dots+f_m}\leq L^m\sum_{k=1}^m\qn{f_k}
\end{equation}
for every $\seqm$ and $f_1,f_2,\dots,f_m\in\calF_0(X)$.
Indeed, if $\mu$ is $K$-relaxed subadditive for some $K\geq 1$, then by the relaxed subadditivity
of the Choquet integral stated in Subsection~\ref{integrals},
for any $f,g\in\calF_0^+(X)$ it follows that
\[
\Ch(\mu,f+g)\leq 2K\bigl\{\Ch(\mu,f)+\Ch(\mu,g)\bigr\}.
\]
Hence, repeated application of the inequality
\[
|a+b|^q\leq 2^q(|a|^q+|b|^q),
\]
which holds for every $a,b\in\bbR$, yields
\begin{equation}\label{eqn:L-sep2}
\qn{f+g}\leq L(\qn{f}+\qn{g}),
\end{equation}
where $L:=2^{1+\frac{2}{q}}K^{\frac{1}{q}}>2$.
Therefore, for any $\seqm$ and $f_1,f_2,\dots,f_m\in\calF_0(X)$,
we may apply~\eqref{eqn:L-sep2} $m$ times to obtain~\eqref{eqn:L-sep1}.
\par
Since $\mu$ has a countable basis of sets in $\calA$ with finite $\mu$-measure,
there is a countable collection $\calD$ of $\calA$-measurable sets with finite $\mu$-measure
such that for any $\ep>0$ and $A\in\calA$ with $\mu(A)<\infty$,
there is a $D\in\calD$ for which $\mu(A\triangle D)<\ep$.
Let $\calE$ be the set of all finite linear combinations of the characteristic functions of sets in $\calD$
with rational coefficients. Then $\calE$ is countable.
If $f\in\calE$, then $f$ can be expressed by
\[
f=\sum_{k=1}^n r_k\chi_{D_k},
\]
where $n\in\bbN$, $r_1,\dots,r_n$ are rational numbers, and $D_1,\dots,D_n\in\calD$.
It thus follows form~\eqref{eqn:L-sep1} that
\[
\qn{f}=\biggl\|\sum_{k=1}^n r_k\chi_{D_k}\biggr\|_q
\leq L^n\sum_{k=1}^n\qn{r_k\chi_{D_k}}
=L^n\sum_{k=1}^n|r_k|\mu(D_k)^{1/q}<\infty,
\]
which shows that $\calE$ is a subset of $\calLq(\mu)$.
\par
Let $f\in\calLq(\mu)$ and $\ep>0$.
Then by Theorem~\ref{L-dense} there is a $\xi\in\calS(X)\cap\calLq(\mu)$ such that
\begin{equation}\label{eqn:L-sep3}
\qn{f-\xi}<\frac{\ep}{3L^3}.
\end{equation}
Since $\xi$ is a simple function, it can be expressed by
\[
\xi=\sum_{k=1}^n c_k\chi_{A_k},
\]
where $\seqn$, $c_1,\dots,c_n\in\bbR$, $A_1,\dots,A_n\in\calA$,
$A_i\cap A_j=\eset\;\,(i\ne j)$, $c_k\ne 0$ for $k=1,2,\dots,n$. Since
\[
|\xi|^q=\biggl|\sum_{k=1}^n c_k\chi_{A_k}\biggr|^q
=\sum_{k=1}^n|c_k|^q\chi_{A_k}\geq |c_k|^q\chi_{A_k},
\]
we have
\[
\qn{\xi}=\Ch(\mu,|\xi|^q)^{1/q}\geq\Ch(\mu,|c_k|^q\chi_{A_k})^{1/q}
=|c_k|\mu(A_k)^{1/q},
\]
which yields $\mu(A_k)<\infty$ for $k=1,2,\dots,n$.
\par
Next take rational numbers $r_1,\dots,r_n$ such that
\begin{equation}\label{eqn:L-sep4}
\max_{1\leq k\leq n}|c_k-r_k|<\frac{\ep}{3L^{n+3}\left(1+\sum_{k=1}^n\mu(A_k)^{1/q}\right)}
\end{equation}
and let
\[
\varphi:=\sum_{k=1}^n r_k\chi_{A_k}.
\]
Since
\[
|\xi-\varphi|=\sum_{k=1}^n|c_k-r_k|\chi_{A_k},
\]
 it follows from~\eqref{eqn:L-sep1} that
\begin{align}\label{eqn:L-sep5}
\qn{\xi-\varphi}
&\leq L^n\sum_{k=1}^n\qn{|c_k-r_k|\chi_{A_k}}\nonumber\\
&=L^n\sum_{k=1}^n|c_k-r_k|\mu(A_k)^{1/q}\nonumber\\
&<L^n\sum_{k=1}^n\mu(A_k)^{1/q}\frac{\ep}{3L^{n+3}
\left(1+\sum_{k=1}^n\mu(A_k)^{1/q}\right)}\nonumber\\
&<\frac{\ep}{3L^3}.
\end{align}
\par
Finally, since $\mu(A_k)<\infty$ for $k=1,2,\dots,n$, take $D_1,\dots,D_n\in\calD$ such that
\begin{equation}\label{eqn:L-sep6}
\max_{1\leq k\leq n}
\mu(A_k\triangle D_k)<\left\{\frac{\ep}{3L^{n+3}(1+\sum_{k=1}^n|r_k|)}\right\}^q
\end{equation}
and let
\[
\psi:=\sum_{k=1}^n r_k\chi_{D_k}.
\]
In the same way as calculating~\eqref{eqn:L-sep5} we obtain
\begin{equation}\label{eqn:L-sep7}
\qn{\varphi-\psi}<\frac{\ep}{3L^3}.
\end{equation}
It thus follows from~\eqref{eqn:L-sep1}, \eqref{eqn:L-sep3},
\eqref{eqn:L-sep5}, and~\eqref{eqn:L-sep7}
that $\qn{f-\psi}<\ep$.
Hence $\calLq(\mu)$ is separable. The separability of $L^q(\mu)$ is now obvious.
\par
The conclusion for $\calLpq(\mu)$ and $L^{p,q}(\mu)$ thus follows since $\mupq$
is conditionally order continuous,
relaxed subadditive, and null-additive if and only if so is $\mu$.
\end{proof}
\section{The Lorentz space of weak type}
In this section, the Lorentz space of weak type is defined by using the Shilkret integral,
and its completeness and dense subsets are considered.
\begin{definition}\label{wL}
Let $\mu\in\calM(X)$. Define the function $\pinftyn{\cdot}\colon\calF_0(X)\to [0,\infty]$ by
\[
\pinftyn{f}:=\Sh(\mu^{1/p},|f|)
\]
for every $f\in\calF_0(X)$ and let
\[
\calL^{p,\infty}(\mu):=\{f\in\calF_0(X)\colon\pinftyn{f}<\infty\}.
\]
If the measure $\mu$ is needed to specify, then $\pinftyn{f}$ is written as $\pinftynm{f}{\mu}$.
Then the space $\calL^{p,\infty}(\mu)$ is called the \textit{Lorentz space of weak type}\/ and
the prenorm $\pinftyn{\cdot}$
is called the \textit{Lorentz prenorm of weak type}\/ on $\calL^{p,\infty}(\mu)$.
\end{definition}
\par
If the prenorm $\pinftyn{\cdot}$ on $\calL^{p,\infty}(\mu)$ is a quasi-seminorm,
then $\calL^{p,\infty}(\mu)$ is a real linear subspace of $\calF_0(X)$,
but this is not the case in general. 
When $\mu$ is $\sigma$-additive, the space $\calL^{p,\infty}(\mu)$
is nothing but the ordinary Lorentz space of weak type (this is also referred
to as the weak $\calL^p$ space);
see~\cite[Theorem~6.6]{C-R}. 
The prenorm $\pinftyn{\cdot}$ does not satisfy the triangle inequality
in general even if $\mu$ is $\sigma$-additive.
\begin{proposition}\label{P-wL}
Let $\mu\in\calM(X)$. Let $0<p<\infty$ and $0<q<\infty$.
\begin{enumerate}
\item[\us{(1)}] For any $A\in\calA$ it follows that $\pinftyn{\chi_A}=\mu(A)^{1/p}$.
\item[\us{(2)}] For any $f\in\calL^{p,\infty}(\mu)$
it follows that $\pinftyn{f}=0$ if and only if $\mu(\{|f|>c\})=0$ for every $c>0$;
they are equivalent to $\mu(\{|f|>0\})=0$ if $\mu$ null-continuous.
\item[\us{(3)}] For any $f\in\calL^{p,\infty}(\mu)$
and $c\in\bbR$ it follows that $\pinftyn{cf}=|c|\pinftyn{f}$.
Hence $\pinftyn{\cdot}$ is homogeneous. 
\item[\us{(4)}] For any $f\in\calL^{p,\infty}(\mu)$
and $c>0$ it follows that $\mu(\{|f|\geq c\})\leq\pinftyn{f}^p/c^p$.
\item[\us{(5)}] For any $f\in\calL^{p,\infty}(\mu)$ it follows that $\pinftyn{f}\leq\pqn{f}$.
\item[\us{(6)}] For any $f,g\in\calL^{p,\infty}(\mu)$,
if $|f|\leq |g|$ then $\pinftyn{f}\leq\pinftyn{g}$.
\item[\us{(7)}] $\mu$ is weakly null-additive if and only if $\pinftyn{\cdot}$ is weakly null-additive.
\item[\us{(8)}] $\mu$ is null-additive if and only if $\pinftyn{\cdot}$ is null-additive.
\item[\us{(9)}] If $\mu$ is $K$-relaxed subadditive for some $K\geq 1$,
then $\pinftyn{\cdot}$ satisfies
the $2^{1+\frac{1}{p}}K^{\frac{1}{p}}$-relaxed triangle inequality.
\item[\us{(10)}] $\mu$ is null-additive if and only if it follows that $\pinftyn{f}=\pinftyn{g}$
whenever $f,g\in\calL^{p,\infty}(\mu)$ and $f\sim g$.
\end{enumerate}
\end{proposition}
\begin{proof}
Assertions (1)--(6) are easy to prove
and assertions (7) and (8) can be derived in the same manner as~\cite[Proposition~3.2]{Kawabe2021}.
\par
(9)\ Let $f,g\in\calL^{p,\infty}(\mu)$. For any $t>0$, we have
\[
\{|f+g|>t\}\subset\{2|f|>t\}\cup\{2|g|>t\},
\]
hence
\begin{equation}\label{eqn:P-wL1}
\mu(\{|f+g|>t\})\leq K\bigl\{\mu(\{2|f|>t\})+\mu(\{2|g|>t\})\bigr\}
\end{equation}
by the $K$-relaxed subadditivity of $\mu$. Then we have
\begin{align*}
\pinftyn{f+g}
&=\sup_{t\in [0,\infty)}t\mu(\{|f+g|>t\})^{1/p}\\
&\leq\sup_{t\in [0,\infty)}t\Big[K\bigl\{\mu(\{2|f|>t\})
+\mu(\{2|g|>t\bigr\})\bigr\}\Bigr]^{1/p}\\[1mm]
&\leq 2^\frac{1}{p}K^\frac{1}{p}\sup_{t\in [0,\infty)}t\bigl\{\mu(\{2|f|>t\})^{1/p}
+\mu(\{2|g|>t\})^{1/p}\bigr\}\\
&\leq 2^\frac{1}{p}K^\frac{1}{p}\biggl(\sup_{t\in [0,\infty)}t\mu(\{2|f|>t\})^{1/p}
+\sup_{t\in [0,\infty)}t\mu(\{2|g|>t\})^{1/p}\biggr)\\[1mm]
&=2^\frac{1}{p}K^\frac{1}{p}\Bigl(\Sh(\mu^{1/p},2|f|)+\Sh(\mu^{1/p},2|g|)\Bigr)\\[1mm]
&=2^{1+\frac{1}{p}}K^\frac{1}{p}\bigl(\pinftyn{f}+\pinftyn{g}\bigr),
\end{align*}
where the first inequality is due to~\eqref{eqn:P-wL1} and the second is due to~\eqref{elementary}.
\par
(10)\ It can be proved in the same manner as (6) of Proposition~\ref{PDS}.
\end{proof}
The quotient space
\[
L^{p,\infty}(\mu):=\{[f]\colon f\in\calL^{p,\infty}(\mu)\}
\]
is defined by the equivalence relation stated in Subsection~\ref{equiv}.
Given an equivalence class $[f]\in L^{p,\infty}(\mu)$,
define the prenorm on $L^{p,\infty}(\mu)$ by $\pinftyn{[f]}:=\pinftyn{f}$,
which is well-defined if $\mu$ is null-additive by (10) of Proposition~\ref{P-wL}.  
This prenorm has the same properties as the prenorm on $\calL^{p,\infty}(\mu)$
and separates points of $L^{p,\infty}(\mu)$, that is, for any $[f]\in L^{p,\infty}(\mu)$,
if $\pinftyn{[f]}=0$ then $[f]=0$.
\par
The following convergence theorems of the Shilkret integral are needed to show the completeness
of $\calL^{p,\infty}(\mu)$ and $L^{p,\infty}(\mu)$ and of interest in themselves.
\begin{proposition}\label{Sh-conv}
Let $\mu\in\calM(X)$. The following assertions are equivalent.
\begin{enumerate}
\item[\us{(i)}] $\mu$ is monotone autocontinuous from below.
\item[\us{(ii)}] The Shilkret monotone nondecreasing almost uniform convergence theorem holds
for $\mu$, that is, for any $\{f_n\}_\seqn\subset\calF_0^+(X)$ and $f\in\calF_0^+(X)$, 
if they satisfy
\begin{enumerate}
\item[\us{(a)}] $f_n(x)\leq f_{n+1}(x)\leq f(x)$ for every $x\in X$ and $\seqn$,
\item[\us{(b)}] $f_n\to f$ $\mu$-a.u.,
\end{enumerate}
then it follows that $\Sh(\mu,f_n)\uparrow\Sh(\mu,f)$.
\item[\us{(iii)}] The Shilkret Fatou almost uniform convergence lemma holds for $\mu$, that is,
for any $\{f_n\}_\seqn\subset\calF_0^+(X)$ and $f\in\calF_0^+(X)$, if $f_n\to f$ $\mu$-a.u.,
then it follows that
\[
\Sh(\mu,f)\leq\liminf_\ninfty\Sh(\mu,f_n).
\]
\end{enumerate}
\end{proposition}
\begin{proof}
(i)$\Rightarrow$(ii)\ We first show the conclusion in the case where $\mu$ is finite.
For each $r>0$, let $g:=f\land r$ and $g_n:=f_n\land r$ for every $\seqn$.
Fix $\seqk$ and take the continuity points $c_1,c_2,\dots,c_k$ of the function $\mu(\{g>t\})$
such that
\begin{itemize}
\item $0=c_0<c_1<c_2<\dots<c_{k-1}<c_k<r$,
\item $|c_i-c_{i-1}|<2r/k\;\,(i=1,2,\dots,k-1)$ and $|r-c_k|<r/k$.
\end{itemize}
For each $\seqn$, let
\[
h_{n,k}:=\bigvee_{i=1}^k c_i\chi_{\{g_n>c_i\}}\quad\mbox{and}\quad 
h_k:=\bigvee_{i=1}^k c_i\chi_{\{g>c_i\}}.
\]
Then $0\leq h_{n,k}(x)\leq r$, $0\leq h_k(x)\leq r$,
$|h_{n,k}(x)-g_n(x)|<2r/k$, and $|h_k(x)-g(x)|<2r/k$ for every $x\in X$ and $\seqn$.
\par
Now, since $\{g_n\}_\seqn$ and $g$ satisfy conditions (a) and (b) of assertion (ii) of Proposition~\ref{Ch-conv},
it follows that
\[
\mu(\{g_n>c_i\})\uparrow\mu(\{g>c_i\})
\]
for every $i\in\{1,2,\dots,k\}$. Hence
\begin{equation}\label{eqn:Sh-conv1}
\Sh(\mu,h_{n,k})=\bigvee_{i=1}^k c_i\mu(\{g_n>c_i\})
\to\bigvee_{i=1}^k c_i\mu(\{g>c_i\})=\Sh(\mu,h_k).
\end{equation}
Since
\[
|\Sh(\mu,g)-\Sh(\mu,h_k)|\leq\frac{2r}{k}\mu(X)\mbox{ and }
|\Sh(\mu,g_n)-\Sh(\mu,h_{n,k})|\leq\frac{2r}{k}\mu(X)
\]
for every $\seqn$, by~\eqref{eqn:Sh-conv1} we have
\[
\limsup_\ninfty|\Sh(\mu,g_n)-\Sh(\mu,g)|\leq\frac{4r}{k}\mu(X).
\]
Since $\seqk$ is arbitrarily fixed, letting $\kinfty$ yields
\[
\Sh(\mu,g)=\lim_\ninfty\Sh(\mu,g_n)=\sup_\seqn\Sh(\mu,g_n).
\]
It thus follows from the upper marginal continuity
of the Shilkret integral stated in Subsection~\ref{integrals} that
\begin{align*}
\Sh(\mu,f)
&=\sup_{r>0}\Sh(\mu,f\land r)=\sup_{r>0}\sup_\seqn\Sh(\mu,f_n\land r)\\
&=\sup_\seqn\sup_{r>0}\Sh(\mu,f_n\land r)=\sup_\seqn\Sh(\mu,f_n).
\end{align*}
\par
We turn to the general case.
For any $s>0$, since the nonadditive measure $\mu\land s$, which is defined by
$(\mu\land s)(A):=\mu(A)\land s$ for every $A\in\calA$, is finite and monotone
autocontinuous from below, it follows from what has been shown above that
\[
\sup_\seqn\Sh(\mu\land s,f_n)=\Sh(\mu\land s,f),
\]
hence
\begin{align*}
\Sh(\mu,f)
&=\sup_{s>0}\sup_\seqn\Sh(\mu\land s,f_n)
=\sup_\seqn\sup_{s>0}\Sh(\mu\land s,f_n)=\sup_\seqn\Sh(\mu,f_n).
\end{align*}
\par
(ii)$\Rightarrow$(iii)\ For each $\seqn$, let $g_n:=\inf_{k\geq n}f_k$.
Then, $\{g_n\}_\seqn$ and $f$ satisfy conditions (a) and (b) of assertion (ii).
It thus follows that
\[
\Sh(\mu,f)
=\lim_\ninfty\Sh(\mu,g_n)\leq\lim_\ninfty
\inf_{k\geq n}\Sh(\mu,f_k)=\liminf_\ninfty\Sh(\mu,f_n).
\]
\par
(iii)$\Rightarrow$(i)\ Take $A,B_n\in\calA$ and assume that $\{B_n\}_\seqn$
is nonincreasing and $\mu(B_n)\to 0$. For each $\seqn$,
let $f_n:=\chi_{A\setminus B_n}$ and $f:=\chi_A$.
Then $f_n\to f$ $\mu$-a.u. Hence, assertion (iii) yields
\begin{align*}
\mu(A)
&=\Sh(\mu,f)\leq\liminf_\ninfty\Sh(\mu,f_n)\\
&=\liminf_\ninfty\mu(A\setminus B_n)\leq\limsup_\ninfty\mu(A\setminus B_n)\leq\mu(A).
\end{align*}
Therefore $\mu$ is monotone autocontinuous from below.
\end{proof}
\par
The proof of the following two results
is left to the reader since the same proof as Theorem~\ref{L-comp1}
and Corollary~\ref{L-comp2} works by using Proposition~\ref{Sh-conv} instead of Proposition~\ref{C-Fatou}.
\begin{theorem}\label{Sh-comp1}
Let $\mu\in\calM(X)$. Let $0<p<\infty$.
Assume that $\mu$ is monotone autocontinuous from below and has property~(C) and the (p.g.p.).
Then $\calL^{p,\infty}(\mu)$ and $L^{p,\infty}(\mu)$ are quasi-complete.
\end{theorem}
\begin{corollary}\label{Sh-comp2}
Let $\mu\in\calM(X)$. Let $0<p<\infty$.
Assume that $\mu$ is relaxed subadditive, monotone autocontinuous from below,
and has property~(C).
Then $\calL^{p,\infty}(\mu)$ is complete with respect to the quasi-seminorm $\pinftyn{\cdot}$
and $L^{p,\infty}(\mu)$ is complete with respect to the quasi-norm $\pinftyn{\cdot}$.
\end{corollary}
\begin{example}
Let $X:=\bbN$ and $\calA:=2^\bbN$.
Let $\mu$ be the nonadditive measure given in Proposition~\ref{counter1}.
Then $\mu$ is subadditive, hence relaxed subadditive, monotone autocontinuous from below,
and has the (p.g.p.), while it does not have property~(C). 
For each $\seqn$, let $A_n:=\{1,2,\dots,n\}$ and $f_n:=\chi_{A_n}$.
Then it follows from (4) of Proposition~\ref{counter1} that $\{f_n\}_\seqn$
does not converge in $\mu$-measure.
Let $0<p<\infty$. Then it follows from (1) of Proposition~\ref{P-wL} that
\[
\pinftyn{f_n}=\left(\sum_{i=1}^n\frac{1}{2^i}\right)^{1/p}\mbox{and}\quad
\pinftyn{f_{n+l}-f_n}=\left(\sum_{i=n+1}^{n+l}\frac{1}{2^i}\right)^{1/p}
\]
for every $n,l\in\bbN$, hence the sequence
$\{f_n\}_\seqn\subset\calL^{p,\infty}(\mu)$ is bounded and Cauchy.
Suppose, contrary to our claim, that $\calL^{p,\infty}(\mu)$ is quasi-complete.
Then, $\{f_n\}_\seqn$ converges with respect to $\pinftyn{\cdot}$, hence it converges
in $\mu$-measure by (4) of Proposition~\ref{P-wL}, which is impossible.
Hence $\calL^{p,\infty}(\mu)$ and $L^{p,\infty}(\mu)$ are not quasi-complete.
This means that property~(C) cannot be dropped
in Theorem~\ref{Sh-comp1} and Corollary~\ref{Sh-comp2}.
\end{example}
According to Theorem~\ref{L-dense}, the set $\calS(X)\cap\calLpq(\mu)$
is dense in $\calLpq(\mu)$.
The following proposition shows that this is not the case of the Lorentz space of weak type.
\begin{proposition}\label{wLex}
Let $X:=(0,1]$ and $\calA$ be the $\sigma$-field of all Borel subsets of $(0,1]$.
Let $\lambda$ be the Lebesgue measure on $\bbR$.
Let $f(x):=1/x$ for every $x\in X$. Then $f\in\calL^{1,\infty}(\lambda)$
and $\calS(X)\subset\calL^{1,\infty}(\lambda)$.
Nevertheless, $\|f-h\|_{1,\infty}\geq 1$ for every $h\in\calS(X)$.
Thus $\calS(X)$ is not dense in $\calL^{1,\infty}(\lambda)$.
\end{proposition}
\begin{proof}
Elementary computation yields $\|f\|_{1,\infty}=1$, hence $f\in\calL^{1,\infty}(\lambda)$.
Take $h\in\calS(X)$ arbitrary and let $r_0:=\max_{x\in X}|h(x)|<\infty$.
Then we have
\[
\|h\|_{1,\infty}=\Sh(\lambda,|h|)\leq\Sh(\lambda,r_0)=r_0<\infty,
\]
hence $\calS(X)\subset\calL^{1,\infty}(\lambda)$.
\par
We proceed to show that $\|f-h\|_{1,\infty}\geq 1$.
For $0\leq r_0\leq 1$, we have
\[
\{(f-r_0)^+>t\}=\begin{cases}
(0,1] & \mbox{if }0\leq t<1-r_0,\\[1mm]
\Bigl(0,\frac{1}{t+r_0}\Bigr) & \mbox{if }t\geq 1-r_0,
\end{cases}
\]
hence $\Sh(\lambda,(f-r_0)^+)=1$.
For $r_0>1$, we have
\[
\{(f-r_0)^+>t\})=\left(0,\frac{1}{t+r_0}\right)
\]
for every $t\geq 0$, hence $\Sh(\lambda,(f-r_0)^+)=1$.
Since $|f-h|\geq (f-r_0)^+$, it follows that
\[
\|f-h\|_{1,\infty}=\Sh(\lambda,|f-h|)\geq\Sh(\lambda,(f-r_0)^+)=1,
\]
which implies that $\calS(X)$ is no longer dense in $\calL^{1,\infty}(\lambda)$.
\end{proof}
\section{The space of the $\bm{\mu}$-essentially bounded functions}\label{EBF}
Unlike the spaces we have considered so far, the space of all $\mu$-essentially bounded functions
is complete under fairly weak conditions.
\begin{definition}\label{Linfty}
Let $\mu\in\calF_0(X)$. Define the function $\inftyn{\cdot}\colon\calF_0(X)\to [0,\infty]$ by
\[
\inftyn{f}:=\inf\{c>0\colon\mu(\{|f|>c\})=0\}
\]
for every $f\in\calF_0(X)$ and let
\[
\calL^\infty(\mu):=\{f\in\calF_0(X)\colon\inftyn{f}<\infty\}.
\]
If the measure $\mu$ is needed to specify, then $\inftyn{f}$ is written as $\inftyn{f\colon\!\mu}$.
The functions in $\calLinfty(\mu)$ are called \emph{$\mu$-essentially bounded functions}.
\end{definition}
\par
When $\mu$ is $\sigma$-additive, the space $\calLinfty(\mu)$ is nothing but
the ordinary seminormed space of all $\mu$-essentially bounded,
$\calA$-measurable functions on $X$ with seminorm $\inftyn{\cdot}$.
In general, the prenorm $\inftyn{\cdot}$ on $\calLinfty(\mu)$ does not satisfy
the triangle inequality.
\begin{proposition}\label{Linftyn}
Let $\mu\in\calM(X)$.
\begin{enumerate}
\item[\us{(1)}] For any $A\in\calA$ it follows that
\[
\inftyn{\chi_A}=\begin{cases}
0 & \mbox{if }\mu(A)=0,\\
1 & \mbox{if }\mu(A)>0.
\end{cases}
\]
\item[\us{(2)}] For any $f\in\calLinfty(\mu)$
it follows that $\inftyn{f}=0$ if and only if $\mu(\{|f|>c\})=0$ for every $c>0$;
they are equivalent to $\mu(\{|f|>0\})=0$ if $\mu$ is null-continuous.
\item[\us{(3)}] For any $f\in\calLinfty(\mu)$ and $c\in\bbR$
it follows that $\inftyn{cf}=|c|\inftyn{f}$.
Hence $\inftyn{\cdot}$ is homogeneous.
\item[\us{(4)}] For any $f,g\in\calLinfty(\mu)$, if $|f|\leq |g|$ then $\inftyn{f}\leq\inftyn{g}$.
\item[\us{(5)}] If $\mu$ is null-continuous, then for any $f\in\calLinfty(\mu)$
it follows that $|f|\leq\inftyn{f}$ $\mu$-a.e.
\item[\us{(6)}] The following assertions are equivalent.
\begin{enumerate}
\item[\us{(i)}] $\mu$ is weakly null-additive.
\item[\us{(ii)}] $\inftyn{\cdot}$ satisfies the triangle inequality.
\item[\us{(iii)}] $\inftyn{\cdot}$ is null-additive.
\item[\us{(iv)}] $\inftyn{\cdot}$ is weakly null-additive.
\end{enumerate}
\item[\us{(7)}] $\mu$ is weakly null-additive if and only if it follows that $\inftyn{f}=\inftyn{g}$
whenever $f,g\in\calLinfty(\mu)$ and $f\sim g$.
\end{enumerate}
\end{proposition}
\begin{proof}
Assertions (1)--(5) are easy to prove.
\par
(6) (i)$\Rightarrow$(ii)\ Let $f,g\in\calLinfty(\mu)$.
For any $a,b\in\bbR$, if $a>\inftyn{f}$ and $b>\inftyn{g}$, then
$\mu(\{|f|>a\})=\mu(\{|g|>b\})=0$, hence $\mu(\{|f+g|>a+b\})=0$
by the weak null-additivity of $\mu$,
and finally $\inftyn{f+g}\leq a+b$.
Letting $a\downarrow\inftyn{f}$ and $b\downarrow\inftyn{g}$
yields $\inftyn{f+g}\leq\inftyn{f}+\inftyn{g}$.
\par
The implications (ii)$\Rightarrow$(iii)$\Rightarrow$(iv) are obvious.
\par
(iv)$\Rightarrow$(i)\ Take $A,B\in\calA$ and assume that $\mu(A)=\mu(B)=0$.
Let $f:=\chi_A$ and $g:=\chi_{B\setminus A}$. Then $\inftyn{f}=\inftyn{g}=0$,
hence $\inftyn{f+g}=0$ by (iv), which shows that $\mu(A\cup B)=0$ since $f+g=\chi_{A\cup B}$.
Therefore, $\mu$ is weakly null-additive.
\par
(7)\ The ``only if'' part:\ Let $f,g\in\calLinfty(\mu)$ and assume that $f\sim g$.
Then $\inftyn{f-g}=0$.
Since $\mu$ is weakly null-additive, $\inftyn{\cdot}$ is null-additive by (6),
so that $\inftyn{f}=\inftyn{g+(f-g)}=\inftyn{g}$.
\par
The ``if'' part:\ Take $A,B\in\calA$ and assume that $\mu(A)=\mu(B)=0$.
Let $f:=\chi_{A\cup B}$ and $g:=\chi_A$. Then $\inftyn{f-g}=\inftyn{\chi_{B\setminus A}}=0$,
hence $f\sim g$. Thus, $\inftyn{f}=\inftyn{g}$, so that $\inftyn{f}=0$ since $\inftyn{g}=0$.
Hence $\mu(A\cup B)=0$, which shows that $\mu$ is weakly null-additive.
\end{proof}
The quotient space
\[
L^{\infty}(\mu):=\{[f]\colon f\in\calL^{\infty}(\mu)\}
\]
is defined by the equivalence relation stated in Subsection~\ref{equiv}.
Given an equivalence class $[f]\in L^{\infty}(\mu)$,
define the prenorm on $L^{\infty}(\mu)$ by $\pinftyn{[f]}:=\pinftyn{f}$,
which is well-defined if $\mu$ is weakly null-additive by (7) of Proposition~\ref{Linftyn}.  
This prenorm has the same properties as the prenorm on $\calLinfty(\mu)$
and separates points of $L^{\infty}(\mu)$, that is, for any $[f]\in L^{\infty}(\mu)$,
if $\inftyn{[f]}=0$ then $[f]=0$.
\par
The following theorem shows that it is not necessary to assume property~(C) and the (p.g.p.)
to show the completeness of the spaces $\calLinfty(\mu)$ and $L^\infty(\mu)$. 
\begin{theorem}\label{comp1}
Let $\mu\in\calM(X)$.
Assume that $\mu$ is weakly null-additive and null-continuous.
Then $\calL^\infty(\mu)$ is complete with respect to the seminorm $\inftyn{\cdot}$
and $L^\infty(\mu)$ is complete with respect to the norm $\inftyn{\cdot}$. 
\end{theorem}
\begin{proof}
Let $\{f_n\}_\seqn\subset\calLinfty(\mu)$ be a Cauchy sequence.
Then, $\mu$ being weakly null-additive, by (6) of Proposition~\ref{Linftyn},
for any $m,n\in\bbN$ we have $f_m-f_n\in\calLinfty(\mu)$, so that
\begin{equation}\label{eq:comp1}
\mu(\{|f_m-f_n|>\inftyn{f_m-f_n}\})=0
\end{equation}
by (5) of Proposition~\ref{Linftyn} and the null-continuity of $\mu$.
Let
\[
E:=\bigcup_{m=1}^\infty\bigcup_{n=1}^\infty\{|f_m-f_n|>\inftyn{f_m-f_n}\}.
\]
For each $\seqk$ let
\[
E_k:=\bigcup_{m=1}^k\bigcup_{n=1}^k\{|f_m-f_n|>\inftyn{f_m-f_n}\}.
\]
Then $E_k\uparrow E$.
Since $\mu$ is weakly null-additive, \eqref{eq:comp1} implies that $\mu(E_k)=0$ for every $\seqk$,
hence $\mu(E)=0$ by the null-continuity of $\mu$.
\par
For any $x\not\in E$, the sequence $\{f_n(x)\}_\seqn$ is Cauchy in $\bbR$,
so that the $\calA$-measurable function $f\colon X\to\bbR$ can be defined by
\[
f(x):=\begin{cases}
\lim\limits_\ninfty f_n(x) & \mbox{if }x\not\in E,\\[1.5mm]
\;0 & \mbox{otherwise}.
\end{cases}
\]
Then it is easy to see that $\inftyn{f_n-f}\to 0$ and $f\in\calL^\infty(\mu)$.
Therefore $\calLinfty(\mu)$ is complete.
Furthermore, $\mu$ being weakly null-additive,
the quotient space $L^\infty(\mu)$ and the prenorm $\inftyn{\cdot}$ on $L^\infty(\mu)$
are well-defined and it turns out that $L^\infty(\mu)$ is complete.
\par
Finally, by (3) and (6) of Proposition~\ref{Linftyn}
the prenorm $\inftyn{\cdot}$ is a seminorm on $\calLinfty(\mu)$ and a norm on $L^\infty(\mu)$.
\end{proof}
The following theorem shows that $S(X)$ is dense in $\calLinfty(\mu)$
for any nonadditive measure $\mu$.
\begin{theorem}
Let $\mu\in\calM(X)$. Then $\calS(X)$ is dense in $\calLinfty(\mu)$.
If $\mu$ is weakly null-additive, then $S(X)$ is dense in $L^\infty(\mu)$.
\end{theorem}
\begin{proof}
Let $f\in\calLinfty(\mu)$. Then there is an $N\in\calA$ such that $\mu(N)=0$
and $f$ is bounded on $X\setminus N$,
thus we can find a sequence $\{h_n\}_\seqn\subset\calS(X)$
such that $\sup_{x\not\in N}|f(x)-h_n(x)|\to 0$.
Let $\ep>0$. Then there is an $n_0\in\bbN$
such that $\sup_{x\not\in N}|f(x)-h_{n_0}(x)|\leq\ep/2$,
so that $\mu(\{|f-h_{n_0}|>\ep/2\})=0$, and finally that $\inftyn{f-h_{n_0}}\leq\ep/2<\ep$.
This means that $\calS(X)$ is dense in $\calLinfty(\mu)$.
The denseness of $S(X)$ in $L^\infty(\mu)$ is now obvious.
\end{proof}
\begin{remark}
As is well-known, the spaces $\calLinfty(\mu)$ and $L^\infty(\mu)$
are not separable even if $\mu$ is $\sigma$-additive.
\end{remark}
\section{Summary of results and future tasks}\label{conclusion}
In this paper, given a nonadditive measure $\mu$,
the space $\calL^0(\mu)$ of all measurable functions,
the Choquet-Lorentz space $\calLpq(\mu)$,
the Lorentz space of weak type $\calL^{p,\infty}(\mu)$,
the space $\calLinfty(\mu)$ of all $\mu$-essentially bounded functions,
and their quotient spaces are defined together with suitable prenorms on them.
The completeness and separability of those spaces are also discussed
relating the characteristic of $\mu$.
Some of our results are as follows.
\begin{itemize}
\item Assume that $\mu$ has property~(C) and the (p.g.p.).
Then $\calL^0(\mu)$ is complete.
If $\mu$ is additionally assumed to be null-additive, then $L^0(\mu)$ is complete.
\item Let $0<p<\infty$ and $0<q\leq\infty$.
If $\mu$ is monotone autocontinuous from below and has property~(C) and the (p.g.p.),
then $\calLpq(\mu)$ and $L^{p,q}(\mu)$ are quasi-complete.
If $\mu$ is relaxed subadditive, monotone autocontinuous from below, and has property~(C),
then $\calLpq(\mu)$ is complete with respect to the quasi-seminorm $\pqn{\cdot}$
and $L^{p,q}(\mu)$ is complete with respect to the quasi-norm $\pqn{\cdot}$.
\item If $\mu$ is weakly null-additive and null-continuous, then $\calLinfty(\mu)$
is complete with respect to the seminorm $\inftyn{\cdot}$ and $L^\infty(\mu)$
is complete with respect to the norm $\inftyn{\cdot}$.
\end{itemize}
\par
All the results listed above hold for every subadditive nonadditive measure
that is continuous from below
since such a nonadditive measure is relaxed subadditive,
monotone autocontinuous from below, null-continuous,
null-additive, weakly null-additive, and has property~(C) and the (p.g.p.).
\par
Concerning dense subsets and the separability,  the following results are shown among others.
\begin{itemize}
\item Assume that $\mu$ is order continuous. Then $\calS(X)$ is dense in $\calL^0(\mu)$.
If $\mu$ is additionally assumed to be null-additive, then $S(X)$ is dense in $L^0(\mu)$.
\item Assume that $\mu$ is order continuous and has the (p.g.p.).
Assume that $\mu$ has a countable basis.
Then $\calL^0(\mu)$ is separable.
If $\mu$ is additionally assumed to be null-additive, then $L^0(\mu)$ is separable.
\item Let $0<p<\infty$ and $0<q<\infty$.
Assume that $\mu$ is conditionally order continuous.
Then $\calS(X)\cap\calLpq(\mu)$ is dense in $\calLpq(\mu)$.
If $\mu$ is additionally assumed to be null-additive, then $S(X)\cap L^{p,q}(\mu)$
is dense in $L^{p,q}(\mu)$.
\item Let $0<p<\infty$ and $0<q<\infty$.
Assume that $\mu$ is conditionally order continuous and relaxed subadditive.
Assume that $\mu$ has a countable basis of sets in $\calA$ with finite $\mu$-measure.
Then $\calLpq(\mu)$ is separable.
If $\mu$ is additionally assumed to be null-additive, then $L^{p,q}(\mu)$ is separable.
\item The set $\calS(X)$ is dense in $\calLinfty(\mu)$ for any nonadditive measure $\mu$.
If $\mu$ is weakly null-additive, then $S(X)$ is dense in $L^\infty(\mu)$.
\end{itemize}
\par
In this manner, the study of function spaces in the framework of
nonadditive measure theory has an advantage of significantly
expanding the scope of application of the theory.
Another advantage is noticeable when studying the Choquet-Lorentz space.
As already mentioned in Section~\ref{CLspace}, for any $\mu\in\calM(X)$ and $f\in\calLpq(\mu)$,
it follows that
\[
\pqnm{f}{\mu}=\left(\frac{p}{q}\right)^{1/q}\qnm{f}{\mupq}
\]
and that $\calLpq(\mu)=\calLq(\mupq)$.
This means that many properties of the Choquet-Lorentz space $\calLpq(\mu)$
can be immediately derived from
those of the space $\calLq(\mupq)$.
Although $\mupq$, which is the power of $\mu$, is not additive in general
even if $\mu$ is additive,
it preserves nonadditive characteristics of $\mu$
such as the weak null-additivity, the null-additivity,
the (p.g.p.), properties~(C), ($\mbox{C}_0$), (S), and ($\mbox{S}_1$),
the relaxed subadditivity, and many kinds of continuity and autocontinuity.
This observation suggests the significance of formulating the results of ordinary measure theory
in terms of the characteristics of nonadditive measures.
This is because the above-mentioned preservation of the characteristic of nonadditive measures
cannot be used when dealing only with additive measures.
\par
The properties established in this paper will be fundamental and important when studying
various function spaces appeared in nonadditive measure theory.
It is a future task to investigate further profound properties of those function spaces.
\par
According to Theorem~\ref{NSC}, when $X$ is countable,
if a nonadditive measure $\mu$ is null-continuous and has the (p.g.p.),
then property~(C) is a necessary and sufficient condition for the Cauchy criterion to hold
for convergence in $\mu$-measure.
It would be desirable to find such a condition for an uncountable $X$,
but we have not been able to do this. This is another future task.
\section*{Acknowledgements}
The authors are grateful to the referees for their helpful comments and suggestions.
%
% References (ref}
%

\end{document}